\documentclass[11pt]{amsart}

\usepackage{hyperref}

\usepackage{xcolor}

\usepackage{amsmath, amssymb, amsfonts, amsthm, fullpage}
\usepackage{graphics,graphicx}
\usepackage{accents}
\usepackage{color}

\usepackage{algorithm}
\usepackage{algorithmicx}
\usepackage{algpseudocode}

\usepackage[margin=0.9in]{geometry}

\newcommand{\R}{\mathbb{R}}

\newcommand{\norma}[1]{{\left\vert\kern-0.25ex\left\vert\kern-0.25ex\left\vert #1
    \right\vert\kern-0.25ex\right\vert\kern-0.25ex\right\vert}}

\makeatletter

\newcommand{\sign}{\mathop{\operator@font sign}}
\makeatother

\newcommand{\by}{\mathbf{y}}

\newcommand{\bF}{\mathbf{F}}

\newcommand{\bw}{\mathbf{w}}

\newcommand{\bv}{\mathbf{v}}
\newcommand{\bA}{\mathbf{A}}

\newcommand{\bD}{\mathbf{D}}

\newcommand{\bG}{\mathbf{G}}

\newcommand{\bJ}{\mathbf{J}}

\newcommand{\bW}{\mathbf{W}}

\newcommand{\tbn}[1]{{\left\vert\kern-0.25ex\left\vert\kern-0.25ex\left\vert #1 \right\vert\kern-0.25ex\right\vert\kern-0.25ex\right\vert}}

\newcommand{\ip}{{i+\frac12}}
\newcommand{\im}{{i-\frac12}}
\newcommand{\xim}{x_{i-\frac12}}
\newcommand{\xip}{x_{i+\frac12}}

\newtheorem{remark}{Remark}[section]
\newtheorem{lemma}{Lemma}[section]
\newtheorem{proposition}{Proposition}[section]
\newtheorem{theorem}{Theorem}[section]

\newtheorem{definition}{Definition}[section]

\AtBeginDocument{%
  \setlength{\belowdisplayskip}{5pt} \setlength{\belowdisplayshortskip}{5pt}
  \setlength{\abovedisplayskip}{5pt} \setlength{\abovedisplayshortskip}{5pt}
}

\begin{document}
\title[A high-order regularization of the non-linear shallow water equations with weakly singular shock waves and its approximation by finite volume methods]{A high-order regularization of the non-linear shallow water equations with weakly singular shock waves and its approximation by finite volume methods}

\author{Rezwana Razzaque Angana}
\address{\textbf{R.R.~Angana:} Victoria University of Wellington, School of Mathematics and Statistics, PO Box 600, Wellington 6140, New Zealand}
\email{rezwana.angana@vuw.ac.nz}

\author{Dimitrios Mitsotakis}
\address{\textbf{D.~Mitsotakis:} Victoria University of Wellington, School of Mathematics and Statistics, PO Box 600, Wellington 6140, New Zealand}
\email{dimitrios.mitsotakis@vuw.ac.nz}

\subjclass[2000]{35Q35, 65M08, 35C07}

\date{\today}


\keywords{weakly singular shock waves, regularized shallow water equations, cusped solitons, finite volume methods}

\begin{abstract}
Considered herein is a high-order regularization of the nonlinear shallow water equations within the framework of water wave theory. The regularized system is Galilean invariant and its solutions maintain an energy level that closely matches that of the nonlinear shallow water equations. However, in contrast to the classical nonlinear shallow water system, which admits discontinuous shock waves, the regularized formulation gives rise to weakly singular shock waves, which have continuous spatial profiles with unbounded spatial derivatives at isolated points. Using dynamical systems techniques, we establish the existence of such waves. Although weakly singular traveling waves remain continuous over their entire domain, their numerical approximation via finite element or pseudospectral schemes is affected by the emergence of spurious oscillations. To address this issue, we explore several finite volume methods for the accurate numerical approximation of these solutions. Our results demonstrate that the regularized system effectively reproduces the dynamics of the nonlinear shallow water equations in several scenarios. Moreover, our computations indicate that weakly singular shock waves are dynamically stable and can arise from general initial conditions connecting two asymptotic states. In contrast, other weakly singular structures, such as cusped solitons, appear to be structurally unstable, as we were unable to generate them from generic initial data.
\end{abstract}

\maketitle

\section{Introduction}

Infinitely long non-dispersive water waves are often described by the Nonlinear Shallow Water Equations \cite{Whitham11}, which in dimensionless and unscaled form can be expressed as:
\begin{equation}\label{eq:sw1}
\begin{aligned}
&\eta_t+[(D+\eta)u]_x=0\ ,\\
&u_t+g\eta_x+uu_x= 0\ ,
\end{aligned}
\end{equation}
where $g$ is the acceleration due to gravity, $y=-D(x)$ represents the bottom topography, $y=\eta(x,t)$ is the free surface elevation above the undisturbed level $y=0$, $u=u(x,t)$ is the depth-averaged horizontal velocity. In the previous notation, $x$, $y$ and $t$ are the horizontal, upward vertical and temporal coordinates, respectively.

Certain waves, as described by nonlinear and non-dispersive hyperbolic conservation laws, eventually break, leading to discontinuous shock waves. To address the challenges arising from the mathematical properties of hyperbolic conservation laws, various regularization techniques have been developed in the literature. The regularization of hyperbolic conservation laws has been explored in several works (cf. e.g., \cite{BL2000,FL2005}). The methodology typically involves adding small perturbations to regularize the discontinuous solutions of the equations.

For instance, to regularize the simple inviscid Burgers equation $u_t+uu_x=0$, one can consider the viscous Burgers equation $u_t+uu_x-\varepsilon u_{xx}=0$, where  $\varepsilon>0$ is small. This technique works well for the Burgers equation, but its solutions are dissipative, eventually causing the ``energy'' of the solution to vanish. To overcome the dissipative nature of the viscous Burgers equation, dispersive regularizations have also been considered.

For example, the KdV equation $u_t+uu_x+\delta u_{xxx}=0$ has been considered as an approximation of the inviscid Burgers equation. However, it is well-known that solutions of the KdV equation with steep gradients tend to evolve into dispersive shock waves as $\delta\to 0$ \cite{GP1987,EHS2017}. In \cite{BS1985}, the dissipative KdV equation $u_t+uu_x+\delta u_{xxx}-\varepsilon u_{xx}=0$ was studied. Because of the highly oscillatory nature of the solutions of the KdV equation for small values of $\delta>0$, it was shown in \cite{BS1985} that for the solution of the KdV-Burgers equation to converge to a shock-wave profile the ratio $\delta/\varepsilon^2$ must be appropriately bounded. 

Contrary to dispersive and dissipative regularizations, non-dispersive and non-dissipative regularizations of the Burgers equation have also been proposed \cite{GJCP2022, BF2006}. The high-order non-dispersive and non-dissipative regularization of the Burgers equation derived in \cite{GJCP2022} results in weakly singular shock waves that are continuous but possess a singular derivative at a single point. These regularized shock waves exhibit behavior similar to classical shock waves, while benefiting from the advantages conferred by continuity.

For the shallow water equations (\ref{eq:sw1}), various regularization techniques have also been developed. These include weakly nonlinear extensions, \cite{BS1985,BCS2002} as well as strongly nonlinear approaches \cite{HFS2022,Lannes2013,S1953I,S1953II}. Direct dispersive extensions of the nonlinear shallow water equations are the weakly nonlinear and weakly dispersive Boussinesq systems of \cite{BCS2002} with flat bottom topography (for extensions of these systems with variable bottom see \cite{IKKM2021, AM2024}), which in dimensionless but scaled variables, these systems can be written as
\begin{equation}\label{eq:bous1}
\begin{aligned}
&\eta_t+[(1+\eta)u]_x-[au_{xxx}+b\eta_{xxt}]=0\ ,\\
&u_t+\eta_x+uu_x-[c\eta_{xxx}+d u_{xxt}]=0\ ,
\end{aligned}
\end{equation}
where $a,b,c,d$, are appropriately chosen real parameters. However, these systems also tend to produce dispersive shock waves. By including sufficiently large dissipation, one can show that, similarly to the KdV-Burgers equation, the dissipative Boussinesq equations can exhibit regularized shock waves,  \cite{BMT2022}.

In the case of strongly nonlinear models such as the Serre-Green-Naghdi equations, it is known that smooth waves with steep gradients evolve into smooth dispersive shock waves \cite{EGS2006}. However, again, with the addition of either dissipative terms, \cite{BMT2023,BMT2024}, or high-order nonlinear and dispersive terms \cite{CD2017,CDM2019}, waves with steep gradients can evolve into regularized shock waves. Specifically, for generalized Serre equations, the inclusion of surface tension effects (in the critical regime with Bond number $1/3$), \cite{MDAZ2017,DMM2018}, or similar high-order Hamiltonian regularization \cite{CD2017, PPDC2018, CDM2019}, allows waves with steep gradients to evolve into weakly singular shock waves. Such solutions are the subject of this work. Note that high-order systems usually serve as good approximations of the Euler equations and have many applications in geophysics, oceanography and coastal hydrodynamics, \cite{KDGS2018,KDFG2020,DM2008}.

In this paper, we explore an alternative non-dispersive and non-dissipative regularization of shallow water equations very similar to those proposed in \cite{CD2017, CDM2019, CDM2024} that are not conservative but simple to demonstrate that weakly singular traveling waves are not a consequence of a Hamiltonian structure but rather an effect of a singularity formed by the inclusion of the high-order terms in combination with the non-dispersive nature of the equations. We also study the approximation of these waves by finite volume methods. Although weakly singular waves are continuous but with unbounded derivative at a singular point, their numerical approximation by finite element or spectral methods is challenging because the singularity in the derivative still causes the generation of spurious oscillations. In \cite{CD2017, CDM2019} pseudo-spectral methods were employed accompanied by filtering technique to eliminate spurious oscillations. In other models, such as the Camassa-Holm equation, finite element and pseudospectral methods appeared to be adequate to capture weakly singular solutions with bounded but discontinuous derivative such as peakons, however, with linear convergence rates, \cite{ADM2019}.

Specifically, in this work we consider the following perturbed shallow water equations
\begin{equation}\label{eq:bous2}
\begin{aligned}
&\eta_t+[(D+\eta)u]_x=0\ ,\\
&u_t+g\eta_x+uu_x=\delta [g\eta_{xxx}+u_{xxt}+ uu_{xxx}]\ ,
\end{aligned}
\end{equation}
where $\delta>0$ is a (small) parameter with dimensions of length squared. This system contains artificial terms from high-order Boussinesq approximations assuming small bottom variations \cite{BCS2002,Chen2003}, which are neglected as insignificant in the  regime of infinitely long shallow water waves. Considering also a small  parameter $\delta$, these artificial terms can be considered as high-order regularization terms that are asymptotically consistent with the particular water waves regime. Furthermore, this system can be considered as a direct extension of the regularized Burgers equation of \cite{GJCP2022} for bi-directional wave propagation.

Boussinesq systems are not generally invariant under vertical translations \cite{dkm2011}, and their dispersion relation depends on the background level of the solution. However, the linear dispersion relation of (\ref{eq:bous2}) with flat bottom $D=const$ around a level $(\eta_0,u_0)$ is given by the relation of the frequency $\omega(k)$ of linear waves as function of the wavenumber $k$
\begin{equation}\label{eq:disprel1}
\omega(k) = k[u_0\pm \sqrt{g(D+\eta_0)}]\ ,
\end{equation}
and coincides with the corresponding linear dispersion relation of (\ref{eq:sw1}). Thus system (\ref{eq:bous}) exhibits non-dispersive and non-dissipative behavior similar to that of the nonlinear shallow water equations (\ref{eq:sw1}).

Furthermore, system (\ref{eq:bous2}) is invariant under the Galilean transformation 
$$x\to x-ct,\quad t\to t,\quad \eta\to \eta, \quad u\to u+c\ .$$
Therefore, it shares these fundamental properties with the shallow water equations and extends the Galilean invariant classical Boussinesq system \cite{DDM2013}.

In this paper, we study system (\ref{eq:bous2}) and demonstrate using dynamical system techniques that when the bottom topography is flat, it admits traveling wave solutions featuring a point singularity, and in particular weakly singular shock waves and cusped solitons. Both types of solutions exhibit an unbounded derivative at a point in space, but with distinct characteristics: cusped solitons correspond to homoclinic orbits to the origin, whereas weakly singular shocks represent heteroclinic orbits connecting two different states. It is worth mentioning that although theoretically cusped solitons may exist as special solutions of this model (see also \cite{PPDC2018}), we failed to generate them in our numerical experiments, indicating that they are unstable.

It has been observed that, for conservative systems, suitable adaptations of multi-symplectic schemes \cite{DCF2013,DDM2019,DDM2019ii} could be advantageous. Nonetheless, as we will show below, the emergence of weakly singular shock waves indicates energy dissipation, and thus dissipative schemes are necessary to control the numerical behavior. Given also that the system (\ref{eq:bous2}) is non-dispersive, we investigate its solutions using finite volume methods. Specifically, we study the approximation properties of the finite volume methods for weakly singular solutions generated by smooth initial conditions that connect two asymptotic states, which may be either different or the same. We compare these solutions with corresponding solutions of the shallow water equations. The effectiveness of finite volume methods for handling smooth nonlinear and dispersive waves has been extensively studied (cf. e.g., \cite{dkm2013,dkm2011,DCMM2013,DDM2018}), particularly in contexts involving wave breaking and long-wave runup. A numerical study of solutions of the generalized Serre-Green-Naghdi equations has been conducted in \cite{PZR2018}, and demonstrated that finite volume methods can be beneficial for studying non-oscillatory waves with steep gradients. It is noted that weakly singular traveling waves are common to high-order nonlinear and dispersive differential equations, and sometimes dynamical systems combined with algebraic techniques can be applied to detect such traveling waves \cite{CDG2014,CDG2015,CDG2016}.

The paper is organized as follows. In Sections~\ref{sec:existence} and~\ref{sec:gwsw} we investigate the existence of traveling-wave solutions of system~(\ref{eq:bous}) from a dynamical-systems perspective. A phase-plane analysis of the profile equations exhibits an interior singularity of the flow and shows how orbits can be concatenated across it to form continuous profiles with an unbounded derivative at a single point, namely weakly singular shock waves and cusped solitons, first for waves connecting to the undisturbed state (Section~\ref{sec:existence}) and then for general two-state configurations (Section~\ref{sec:gwsw}). In Section~\ref{sec:weak}, we rigorously formalize this construction and show that the concatenated orbits, also referred to as glued profiles, constitute weak (distributional) solutions of the traveling-wave equations, even though their derivatives exhibit a singularity. Section~\ref{sec:conservation} shows that, for small $\delta$ the energy of the regularized system stays close to that of the nonlinear shallow water equations, with a deviation of order $O(\sqrt{\delta}\,t)$ even in the presence of weakly singular traveling waves. Section~\ref{sec:finitevols} develops finite volume methods for the discretization of (\ref{eq:bous}) and Section \ref{sec:accuracy} focuses on the accuracy of the method for the approximation of weakly singular shock waves. Section \ref{sec:numericalexp} reports numerical experiments that assess these methods and examine the generation, interaction, and stability of weakly singular solutions, including their behavior over a variable bottom. Finally, Section~\ref{sec:conclusions} summarizes our findings and discusses their implications.

\section{weakly singular traveling wave solutions}\label{sec:existence}

In this section, we investigate potential traveling wave solutions of the regularized system (\ref{eq:bous2}). For simplicity, we consider (\ref{eq:bous2}) with $g=D=1$ and $\delta>0$ in the form 
\begin{equation}\label{eq:bous}
\begin{aligned}
&\eta_t+[(1+\eta)u]_x=0\ ,\\
&u_t+\eta_x+uu_x-\delta [\eta_{xxx}+u_{xxt}+ uu_{xxx}]=0\ .
\end{aligned}
\end{equation}
We search for traveling waves with phase speed $s>1$ in the form 
\begin{equation}\label{eq:ansatz}
\eta(x,t)=\eta(x-s t)\qquad \text{and}\qquad u(x,t)=u(x-s t)\ ,    
\end{equation}
such that
\begin{equation}\label{eq:bcs}
\begin{aligned}
&\lim_{\xi\to-\infty} \eta(\xi)=\eta_{-} \quad\text{and} \quad \lim_{\xi\to-\infty} u(\xi)=u_{-}\\
&\lim_{\xi\to+\infty} \eta(\xi)=\eta_{+} \quad \text{and} \quad \lim_{\xi\to+\infty} u(\xi)=u_{+}\ ,  
\end{aligned}
\end{equation}
where $\xi=x-s t$ defines a coordinate system relative to the traveling wave's frame of reference. For simplicity we consider here 
\begin{equation}\label{eq:bcs2}
\eta_{-}>0, \quad u_{-}>0,\quad \eta_{+}=0\quad \text{and}\quad u_{+}=0\ .
\end{equation}
Due to the symmetry ($\xi\to -\xi$, $s\to -s$), the roles of $\eta_{\pm}$ and $u_{\pm}$ can be interchanged, thereby making the analysis of left-traveling waves analogous to that of right-traveling waves. Additionally, we assume that all the derivatives of the solution vanish at infinity, i.e. $$\lim_{|\xi|\to \infty} \frac{d^n}{dx^n}\eta(\xi)=\lim_{|\xi|\to \infty}\frac{d^n}{dx^n}u(\xi)=0 \quad \text{for all $n=1,2,\dots$}\ .$$

Substituting the {\em ansatz} (\ref{eq:ansatz}) into (\ref{eq:bous}) yields the system of ordinary differential equations
\begin{equation}\label{eq:systemd}
\begin{aligned}
&-s \eta'+[(1+\eta)u]'=0\ ,\\
&-s u'+\eta'+\tfrac{1}{2}(u^2)'-\delta[\eta''-su''+uu''-\tfrac{1}{2}(u')^2]'=0\ ,
\end{aligned}
\end{equation}
where we used $uu'''=(uu'')'-\tfrac{1}{2}[(u')^2]'$.
Integration of (\ref{eq:systemd}) over the interval $(\xi,+\infty)$ subject to the boundary conditions (\ref{eq:bcs}), (\ref{eq:bcs2}), and elimination of the unknown $u$ leads to the following ordinary differential equation:
\begin{equation}\label{eq:ode1}
\delta \eta''=\eta+\frac{s^2}{2(s^2-(1+\eta)^3)}\left[\eta^2(3+\eta)+3\delta \frac{(\eta')^2}{1+\eta}\right]\ ,
\end{equation}
while 
\begin{equation}
u=\frac{s\eta}{1+\eta}\ .
\end{equation}

To study (\ref{eq:ode1}), we introduce the auxiliary variable $h=\delta \eta'$ transforming it into a system of differential equations:
\begin{equation}\label{eq:sode1}
\begin{aligned}
\eta' &= h/\delta\ ,\\
h' &= \eta+\frac{s^2}{2(s^2-(1+\eta)^3)}\left[\eta^2(3+\eta)+\frac{3}{\delta} \frac{h^2}{1+\eta}\right]\ .
\end{aligned}
\end{equation}
This system can be expressed in the form $\by'=\bF(\by)$, where 
$$
\by=\begin{pmatrix}\eta\\h\end{pmatrix}\quad \text{and}\quad
\bF(\by)=\begin{pmatrix}
    h/\delta\\
    \eta+\frac{s^2}{2(s^2-(1+\eta)^3)}\left[\eta^2(3+\eta)+\frac{3}{\delta} \frac{h^2}{1+\eta}\right]
\end{pmatrix}\ .$$
It is important to observe that for $\eta(\xi) = \sqrt[3]{s^2} - 1$, the system (\ref{eq:sode1}) develops a singularity at which its right-hand side is no longer well-defined.

The equilibrium points of system (\ref{eq:sode1}) are $(\eta_i,0)$ for $i=-1,0,1,2$, where
\begin{equation}\label{eq:criticals}
\eta_{-1}=-1,\quad \eta_0=0,\quad \eta_1=\frac{s^2-4+ s\sqrt{s^2+8}}{4}\quad\text{and}\quad \eta_2=\frac{s^2-4- s\sqrt{s^2+8}}{4}\ .
\end{equation}
The critical point $\eta_1$ is positive for $s>1$. However, the critical points $\eta_{-1}$ and $\eta_2$ are negative for $s>1$ and they are excluded from our analysis. It is noteworthy that $\eta_0<\sqrt[3]{s^2}-1$, while $\eta_1>s-1>\sqrt[3]{s^2}-1$. Additionally, the critical point $\eta_1$ coincides with the corresponding critical point of the dissipative Peregrine-Boussinesq system (and of the shallow water equations), \cite{BMT2022, BMT2023}. Moreover, taking the limits as $\xi\to\pm\infty$ in system (\ref{eq:sode1}), we find $\eta_{+}=\eta_0$ and $\eta_{-}=\eta_1$.

To analyze the stability properties of the critical points $(\eta_i,0)^T$, we consider the corresponding linearizations
$$
\begin{pmatrix}
\eta-\eta_i\\
h
\end{pmatrix}'=
\begin{pmatrix}
0 & 1/\delta\\
F'(\eta_i,0) &0
\end{pmatrix}
\begin{pmatrix}
\eta-\eta_i\\
h
\end{pmatrix}\ ,
$$
where
$$F'(\eta_i,0)=1+\frac{3\eta_i^2(1+\eta_i)^2(3+\eta_i)s^2}{2(s^2-(1+\eta_i)^3)^2}+\frac{s^2(\eta_i^2+2\eta_i(3+\eta_i))}{2(s^2-(1+\eta_i)^3)}\ .$$
Therefore, the Jacobian matrices at the critical points $(\eta_0,0)$ and $(\eta_1,0)$ coincide and are given by:
$$\bJ(\eta_0,0)=\bJ(\eta_1,0)=\begin{pmatrix}
0 & 1/\delta\\
1 & 0
\end{pmatrix}\ .$$
The eigenvalues of the Jacobians are $\lambda_{\pm}=\pm 1/\sqrt{\delta}$.
Therefore, both critical points $(\eta_0,0)$ and $(\eta_1,0)$ are saddle points. 

These critical points represent the limit-sets of orbit solutions of system (\ref{eq:sode1}). However, due to the asymptote $\eta=\sqrt[3]{s^2}-1$ in the phase space, they cannot be connected by a continuous orbit (see Figure \ref{fig:phase1}). We define four regions based on the separation of the flow by the lines $\eta=\sqrt[3]{s^2}-1$ and $h=0$. Specifically, we define the regions:
$$
\begin{aligned}
&Q_1=\{(\eta,h)~:~ h>0,~ 0<\eta<\sqrt[3]{s^2}-1 \}, &Q_2=\{(\eta,h)~:~ h>0,~ \sqrt[3]{s^2}-1<\eta<\eta_{+}\}\ ,\\
&Q_3=\{(\eta,h)~:~ h<0,~ 0<\eta<\sqrt[3]{s^2}-1 \}, &Q_4=\{(\eta,h)~:~ h<0,~ \sqrt[3]{s^2}-1<\eta<\eta_{+}\}\ .\\
\end{aligned}
$$

\begin{figure}[t]
\centering
\includegraphics[width=0.8\columnwidth]{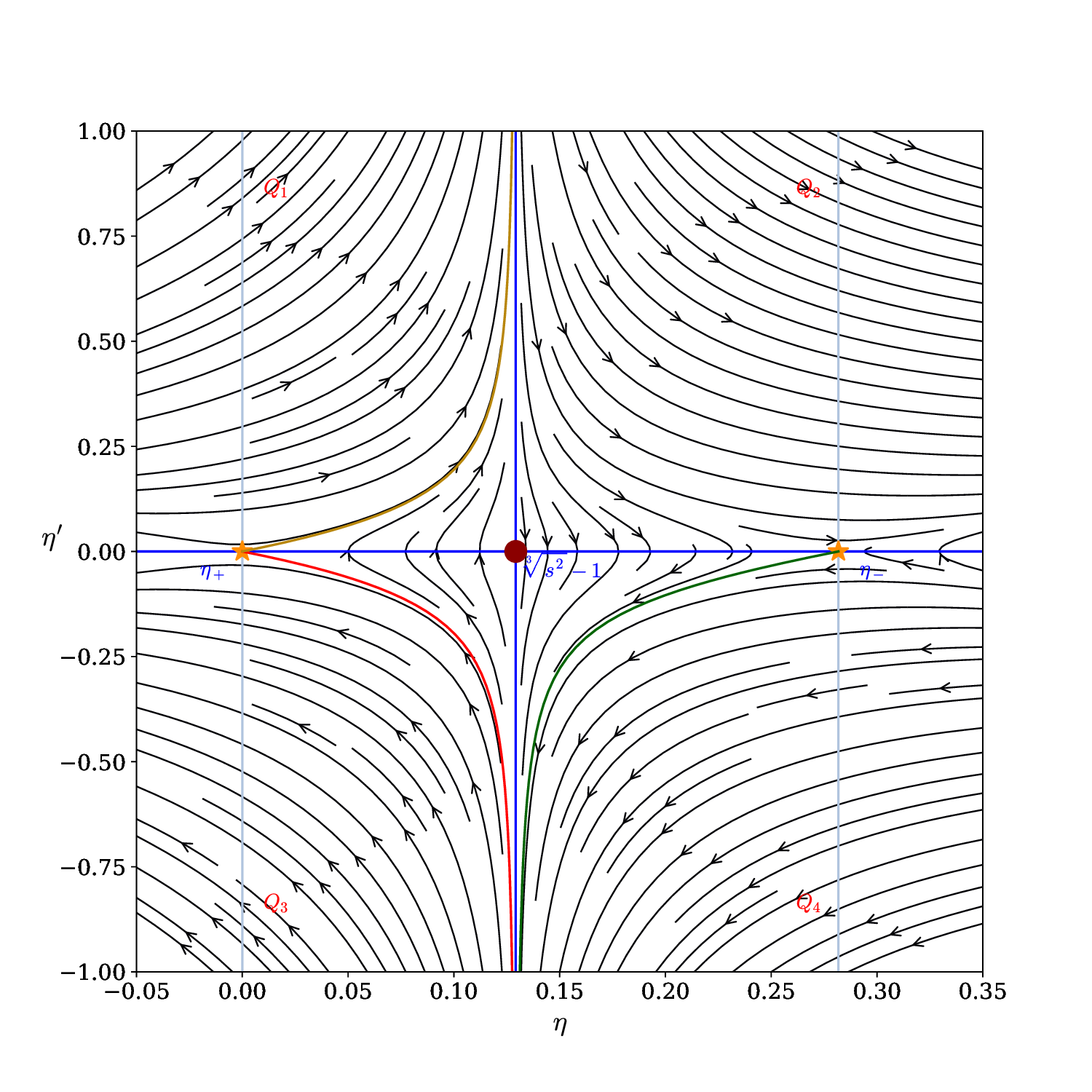}
\caption{Phase-space diagram, orbits and flow steamlines in the strip $\eta_{+}<\eta<\eta_{-}$ for $s=1.2$ and $\delta=1$}
\label{fig:phase1}
\end{figure}

The stable and unstable manifolds corresponding to the critical points $(\eta_0,0)$ and $(\eta_1,0)$ are tangent to the directions $(1,\pm\sqrt{\delta})$. Therefore, as $\delta\to 0$, the stable and unstable manifolds tend towards the $\eta$-axis. This behavior implies that orbits departing or approaching the critical points do so smoothly.

\begin{figure}[t]
\centering
\includegraphics[width=0.8\columnwidth]{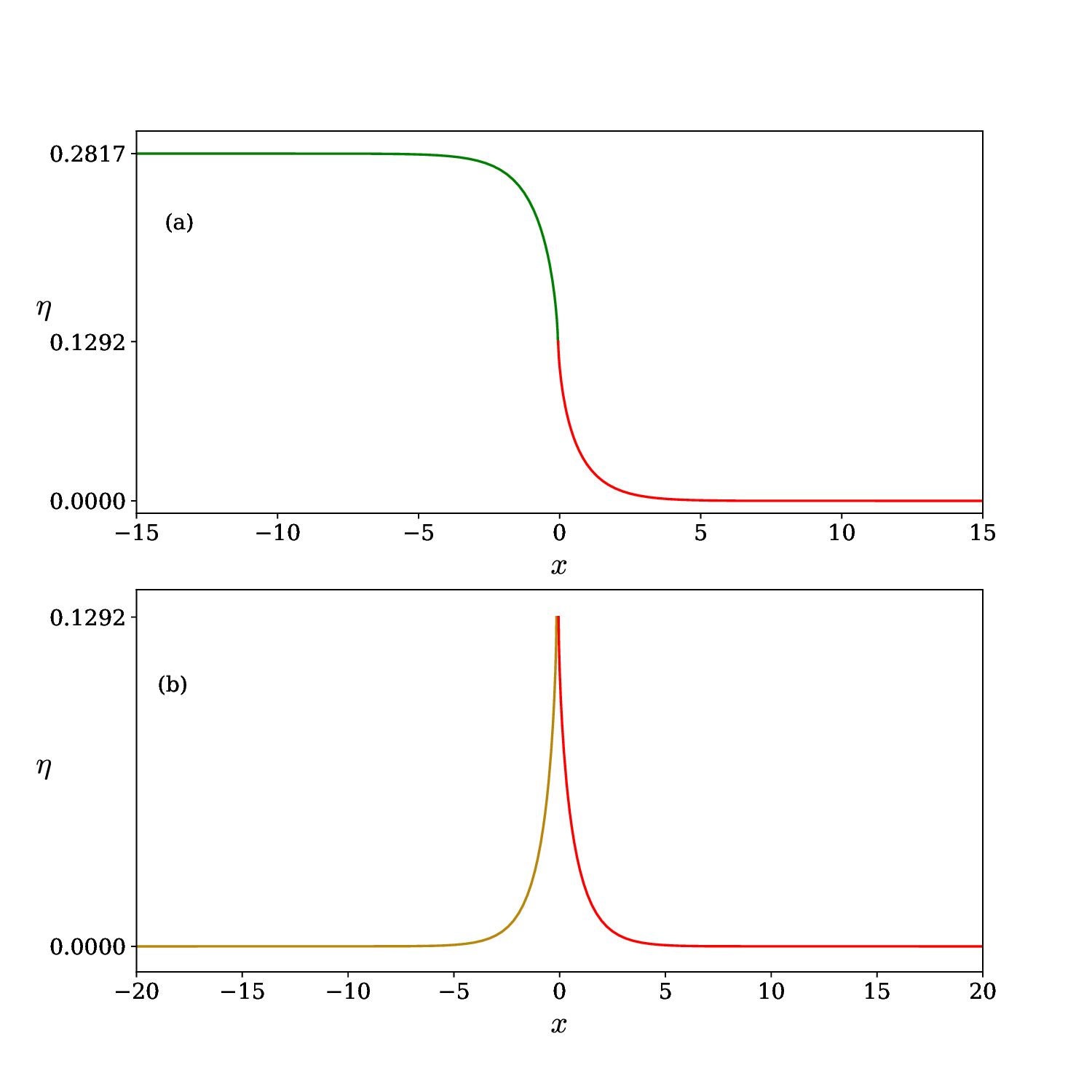}
\caption{(a) weakly singular shock solution and (b) cusped solutions of (\ref{eq:sode1}) for $s=1.2$ and $\delta=1$. Different colors correspond to different orbits}
\label{fig:rshock0}
\end{figure}

Examining the phase diagram in Figure \ref{fig:phase1}, we observe that an orbit with initial conditions in region $Q_1$, where $0 < \eta < \sqrt[3]{s^2} - 1$ and $\eta' > 0$, will remain confined to $Q_1$. This is because within this region we have $h' > 0$ from (\ref{eq:sode1}), implying that $h$ is increasing. As a result, the solution $\eta$ grows monotonically and asymptotically approaches the limiting value $\eta_\ast=\sqrt[3]{s^2} - 1$. 

Similarly, any orbit in $Q_4$ with $\sqrt[3]{s^2} - 1<\eta<\eta_{-}$, will have $h'<0$ from (\ref{eq:ode1}). Thus $h$ will remain decreasing and negative. As a result, the solution $\eta$ will be decreasing monotonically and eventually will approach the limiting value $\sqrt[3]{s^2} - 1$. For the regions $Q_2$ and $Q_3$ we can draw the same conclusions working in reverse time.
 
Because of the nature of the singularity, this will occur in finite time. In other words, it is impossible for $\lim_{\xi\to+\infty}\eta(\xi)=\sqrt[3]{s^2} -1$ and $\lim_{\xi\to+\infty}\eta'(\xi)=+\infty$. To see this, assume for contradiction that this is true, and take large $T>0$ such that $\eta'(\xi)>1$ for all $\xi>T$. Integration of the last inequality over any interval $[T,\xi]$ will give $\eta(\xi)>\eta(T)+\xi-T$. Taking $\xi\to \infty$, gives $\eta(\xi)\to\infty$. Due to the invariance of the traveling waves solutions to horizontal translations, we can assume without loss of generality that the singularity occurs at $\xi_0=0$.

Consider the orbit that tends to $(\eta_{-},0)$ as $\xi \to -\infty$ and asymptotically approaches the value $\eta_\ast=\sqrt[3]{s^2}-1$ at $\xi_0=0$. Accordingly, this orbit follows the unstable manifold in region $Q_1$ as $\xi \to -\infty$ and approaches $\eta_\ast = \sqrt[3]{s^2}-1$ as $\xi \to 0^{-}$. The green curve in Figure \ref{fig:phase1} represents this orbit, computed numerically for $s=1.2$ over the interval $[-21,-5\times 10^{-6}]$ using the adaptive Runge–Kutta method RK45, as implemented in Python’s \texttt{solve\_ivp} function \cite{Mitsotakis2023}, with the initial condition $(10^{-10},10^{-10})$. Owing to numerical limitations, however, the computation extended only up to $\xi = -0.0525$, thus yielding a half-orbit defined for $\xi < 0$.

Similarly, an orbit arriving at $(\eta_{+},0)=(0,0)$ in $Q_3$ with $0<\eta<\sqrt[3]{s^2}-1$ and $h<0$ will remain in $Q_3$ for all $\xi\in (0,+\infty)$. This follows from the second equation of (\ref{eq:sode1}), which shows that $h' > 0$ within $Q_3$. By arguments analogous to those used previously, any orbit in $Q_3$ will asymptotically approach the value $\sqrt[3]{s^2} - 1$ as $\xi \to 0^{+}$.

The red curve in Figure \ref{fig:phase1} corresponds to this orbit, computed numerically for $s=1.2$ over the interval $[5\times 10^{-6},19]$ in reverse time using the adaptive Runge–Kutta method RK45, as implemented in Python’s \texttt{solve\_ivp} function, with the initial condition $(\eta_{+} + 10^{-9},10^{-9})$. Due to numerical constraints, however, the computation extended only up to $\xi = 0.0525$, thus yielding a half-orbit defined for $\xi > 0$.

By the symmetry $(x,t) \mapsto (x,-t)$ and $(\eta,u) \mapsto (\eta,-u)$, counterparts of the orbits found in regions $Q_3$ and $Q_4$ also exist in regions $Q_1$ and $Q_2$. In region $Q_1$, these orbits extend for $\xi < 0$, while in region $Q_2$ they extend for $\xi > 0$.

An orbit departing from $(\eta_{-},0)$ satisfies the boundary condition $\lim_{\xi \to -\infty} \eta(\xi) = \eta_{-}$ and approaches the value $\sqrt[3]{s^2} - 1$ as $\xi \to 0^{-}$. Conversely, an orbit tending to the point $(\eta_{+},0)$ satisfies the boundary condition $\lim_{\xi \to +\infty} \eta(\xi) = \eta_{+} = 0$ and approaches the same value $\sqrt[3]{s^2}-1$ as $\xi \to 0^{+}$.

The yellow orbit in Figure \ref{fig:phase1} represents another half-orbit, extending for $\xi < 0$. It was numerically computed using the same method as before, with $s = 1.2$ and $\delta = 1$, over the interval $[-21,5 \times 10^{-6}]$. The initial condition for this computation was $(10^{-10}, 10^{-10})$.

Orbits from different regions of the strip $\eta_{+} < \eta < \eta_{-}$ can be combined to construct single solutions of (\ref{eq:sode1}), which satisfy the equation in a weak sense. For example, the concatenation of the red orbit from region $Q_3$ with the green orbit from region $Q_4$ produces a weakly singular shock solution, as shown in Figure \ref{fig:rshock0}(a) for $s = 1.2$ and $\delta = 1$. Similarly, synthesizing the yellow orbit from region $Q_1$ with the red orbit from region $Q_3$ results in a cusped solution of (\ref{eq:sode1}), also illustrated in Figure \ref{fig:rshock0}(b). The proof of existence of such solutions using functional analysis techniques can also be demonstrated following the same methodology as in \cite{PPDC2018}.

Traveling waves of (\ref{eq:ode1}) can also be computed more efficiently using minimization techniques, such as the Newton method with the Levenberg–Marquardt modification \cite{DDM2018}. In Section \ref{sec:numericalexp}, we employ this method to investigate weakly singular shock solutions arising from general initial conditions.

\section{General weakly singular shock waves}\label{sec:gwsw}

Consider the boundary conditions (\ref{eq:bcs}) with general $\eta_{-}>\eta_{+}$ and $u_{-}>u_{+}$. Integration of system (\ref{eq:systemd}) in $(\xi,\infty)$ yields
\begin{equation}\label{eq:systemd2}
\begin{aligned}
&-s\eta+(1+\eta)u=A\ ,\\
&-su+\eta+\tfrac{1}{2}u^2-\delta(\eta''-su''+uu''-\tfrac{1}{2}(u')^2)=B\ ,
\end{aligned}
\end{equation}
where 
\begin{equation}\label{eq:k1}
A=-s\eta_{-}+(1+\eta_{-})u_{-}=-s\eta_{+}+(1+\eta_{+})u_{+}\ ,
\end{equation}
and
\begin{equation}\label{eq:k2}
B=-su_{-}+\eta_{-}+\tfrac{1}{2}u_{-}^2=-su_{+}+\eta_{+}+\tfrac{1}{2}u_{+}^2\ .
\end{equation}
In this setting, the Lax condition requires
$$u_{+}+\sqrt{1+\eta_{+}}<s<u_{-}+\sqrt{1+\eta_{-}}\ ,$$
while from (\ref{eq:k1}) and (\ref{eq:k2}) we have
\begin{equation}\label{eq:equilib}
s=\frac{(1+\eta_{-})u_{-}-(1+\eta_{+})u_{+}}{\eta_{-}-\eta_{+}}\quad \text{and}\quad s=\frac{\eta_{-}+\tfrac{1}{2}u_{-}^2-\eta_{+}-\tfrac{1}{2}u_{+}^2}{u_{-}-u_{+}}\ .
\end{equation}
These conditions coincide with the corresponding shock conditions of the classical nonlinear shallow water equations (\ref{eq:sw1}), as well as the Serre-Green-Naghdi and Peregrine systems \cite{BMT2024,BMT2022}.

Setting $A=C+s$ and $B=G-\tfrac{1}{2}s^2$, (\ref{eq:systemd2}) yields
\begin{equation}\label{eq:systemd3}
\begin{aligned}
& u=s+\frac{C}{1+\eta}  \ ,\\
& \eta+\tfrac{1}{2}(u-s)^2-\delta(\eta''+(u-s)u''-\tfrac{1}{2}(u')^2)=G\ ,
\end{aligned}
\end{equation}
where $C=(1+\eta_{+})(u_{+}-s)$ and $G=\eta_{+}+\tfrac12(u_{+}-s)^2$.
Substitution of the first equation of (\ref{eq:systemd3}) into the second equation leads to the second order differential equation:
\begin{equation}\label{eq:systemd4}
\delta \eta''=\eta+\frac{C^2}{2(C^2-(1+\eta)^3)}\left[\sigma (1+\eta)^3-(1+3\eta)+3\delta \frac{(\eta')^2}{(1+\eta)}\right]\ ,
\end{equation}
where
$$\sigma=\frac{2\eta_{+}+(s-u_{+})^2}{(1+\eta_{+})^2(s-u_{+})^2}\ .$$
In the case where $\eta_{+}=u_{+}=0$, the equation (\ref{eq:systemd4}) coincides with (\ref{eq:ode1}). Moreover, the analysis of the previous section applies identically to the equation (\ref{eq:systemd4}). In the general case, the singularity occurs when $\eta=\eta_\ast$ where
$$\eta_\ast:=\sqrt[3]{C^2}-1=\sqrt[3]{[(u_{+}-s)(1+\eta_{+})]^2}-1 \ .$$ 
The critical points satisfy the equation
\begin{equation}\label{eq:cubiceq}
(1 + \eta) (-2 \eta (1 + \eta)^2 + C^2 (-1 + (1 + \eta)^2 \sigma))=0\ .
\end{equation}
Thus, again one critical point is $\eta_{-1}=-1$, while the rest depend on the values of $\eta_{\pm}$ and $u_{\pm}$.

Specifically, we obtain 
\begin{equation}
\eta_0=\eta_{+},\quad  \eta_1=-1 + \frac{1}{4} (s - u_{+}) (s + \sqrt{8(1+\eta_{+}) + (s - u_{+})^2} - u_{+})\ ,
    \end{equation}
and $\eta_2=-1 - \tfrac{1}{4} (s - u_{+}) (-s + \sqrt{8(1+\eta_{+}) + (s - u_{+})^2} + u_{+})$, where the equilibria $\eta_{-1}$, and $\eta_2$ are excluded for physically meaningful solutions as before. The value $\eta_1$ is the value of $\eta_{-}$ we obtain after eliminating $u_{-}$ from (\ref{eq:equilib}) and solving for $\eta_{-}$. These  are the same equilibria one can obtain for the shallow water equations (\ref{eq:sw1}).

In Section \ref{sec:numericalexp} we present the formation of general weakly singular shock waves as a result of the head-on collision of two weakly singular shock waves traveling in opposite directions.

\section{Existence of weakly singular traveling waves}\label{sec:weak}

Throughout, $\mathcal{D}(\mathbb{R})=C_c^{\infty}(\mathbb{R})$ denotes the space of smooth, compactly supported test functions and $\mathcal{D}'(\mathbb{R})$ its dual, the space of distributions, with duality pairing $\langle\cdot,\cdot\rangle$. For $1\le p\le\infty$, $L^p_{\mathrm{loc}}(\mathbb{R})$ is the space of measurable
functions $f$ with $\int_K|f|^p\, d\xi<\infty$ for every compact $K\subset\mathbb{R}$, and $W^{k,p}_{\mathrm{loc}}(\mathbb{R})$ the space of $f\in L^p_{\mathrm{loc}}(\mathbb{R})$ whose distributional derivatives up to order
$k$ also belong to $L^p_{\mathrm{loc}}(\mathbb{R})$. We write $C(\mathbb{R})$ and $C^1(\mathbb{R})$ for the continuous and continuously differentiable functions, and, for $0<\alpha\le1$, $C^{0,\alpha}_{\mathrm{loc}}(\mathbb{R})$ for the locally $\alpha$-H\"older continuous functions, i.e.\ those $f$ for which
$\sup\{|f(\xi)-f(\zeta)|/|\xi-\zeta|^{\alpha}:\xi\neq\zeta\in K\}<\infty$ on every compact $K$. A prime denotes $d/ d\xi$. We also write $o(1)$ for a quantity tending to $0$ as $\xi\to0$, and $f\sim g$ when $f/g\to1$.

In Sections \ref{sec:existence} and \ref{sec:gwsw}, we patched half-orbits of the planar system (\ref{eq:sode1}) and asserted that the result is a traveling wave in a weak sense. We now make this precise.

Let
\begin{equation}\label{eq:Pdefg}
P:=\eta+\tfrac{1}{2}(u-s)^2\ ,
\end{equation}
is the momentum flux.
Using $u=s+\tfrac{C}{1+\eta}$ and setting
\begin{equation}\label{eq:Qidg}
Q=\eta''+(u-s)u''-\tfrac{1}{2}(u')^2=P''-\tfrac{3}{2}(u')^2\ ,
\end{equation}
the second equation of the momentum balance (\ref{eq:systemd3}) becomes
\begin{equation}\label{eq:starfluxg}
P-\delta P''+\tfrac{3}{2}\delta\,(u')^2=G\ ,\qquad G:=\eta_{+}+\tfrac12(u_{+}-s)^2\ ,
\end{equation}
where now $P''$ is the (distributional) second derivative of a continuous function and $(u')^2$ is a nonnegative locally integrable function. Differentiating (\ref{eq:starfluxg}) and using $-su'+\eta'+uu'=[P-\tfrac12 s^2]'$ together with (\ref{eq:Qidg}) recovers exactly the momentum equation of (\ref{eq:bous}) under the traveling {\em ansatz}, with the otherwise ill-defined term $\delta\,uu_{xxx}$ realized through the combination $Q=P''-\tfrac32(u')^2$.

We will denote the singular value by $\eta_{\ast}:=\sqrt[3]{C^2}-1$. Recall that $C=(1+\eta_{+})(u_{+}-s)$ and $\sigma$ are the constants of (\ref{eq:systemd3})--(\ref{eq:systemd4}), and we write $C^2=(1+\eta_{+})^2(s-u_{+})^2=(1+\eta_{-})^2(s-u_{-})^2$ (the second equality follows from (\ref{eq:k1})). The canonical case $\eta_{+}=u_{+}=0$ of Section \ref{sec:existence} corresponds to $C=-s$, $\sigma=1$, $C^2=s^2$, and is recovered throughout by these substitutions.

First we show the existence of an explicit first integral:
\begin{lemma}\label{lem:firstintg}
Equation (\ref{eq:systemd4}) has the form $\delta\eta''=R(\eta)+S(\eta)\cdot (\eta')^2$ and
admits the first integral
\begin{equation}\label{eq:fig}
(\eta')^2=\frac{(1+\eta)^3}{\,C^2-(1+\eta)^3\,}\left(\frac{\Psi(\eta)}{\delta}+E\right),
\qquad
\Psi(\eta):=(2+\sigma C^2)(1+\eta)-(1+\eta)^2+\frac{C^2}{1+\eta}\ ,
\end{equation}
with $E$ constant along each orbit. On the half-orbit limiting on the equilibrium
$\eta_{\mathrm e}\in\{\eta_{+},\eta_{-}\}$ (where $\eta'=0$) one has $E=-\Psi(\eta_{\mathrm e})/\delta$,
so that
\begin{equation}\label{eq:fi2g}
(\eta')^2=\frac{(1+\eta)^3}{\delta\,(C^2-(1+\eta)^3)}\,\left(\Psi(\eta)-\Psi(\eta_{\mathrm e})\right)\ .
\end{equation}
Moreover the critical points of $\Psi$ are exactly the equilibria of (\ref{eq:systemd4}); in the physical range $\eta>-1$ these are $\eta_{+}$ and $\eta_{-}$, and $\Psi$ is strictly monotone on $(\eta_{+},\eta_{-})$.
\end{lemma}

\begin{proof}
We first write (\ref{eq:systemd4}) as 
\begin{equation}\label{eq:systemd42}
\delta\eta''=R(\eta)+S(\eta)\cdot(\eta')^2\ ,
\end{equation} 
with
\begin{equation}\label{eq:RSg}
R(\eta)=\eta+\frac{C^2\big[\sigma(1+\eta)^3-(1+3\eta)\big]}{2\left(C^2-(1+\eta)^3\right)},\quad
S(\eta)=\frac{3\delta C^2}{2\left(C^2-(1+\eta)^3\right)(1+\eta)}\ .    
\end{equation}

Multiplying (\ref{eq:systemd42}) by $\mu(\eta):=\tfrac{C^2-(1+\eta)^3}{(1+\eta)^3}$ casts it in
Euler--Lagrange form 
\begin{equation}\label{eq:lagrangeform}
\tfrac{d}{ d\xi}\partial_{\eta'}\mathcal L-\partial_\eta\mathcal L=0\ ,
\end{equation}
where 
\begin{equation}\label{eq:lagrangiani}
\mathcal L(\eta,\eta')=\tfrac12\,\mu(\eta)(\eta')^2-V(\eta)\ ,
\end{equation}
with
\begin{equation}\label{eq:Vprime}
V'(\eta)=-\tfrac1\delta\,\mu(\eta)R(\eta)\ ,
\end{equation}
and the choice of $\mu$ being fixed by $\mu'=-\tfrac2\delta\mu S$.
It is straightforward also to verify that the Beltrami energy
\begin{equation}\label{eq:beltrami}
\mathcal{H}=\eta'\partial_{\eta'}\mathcal L-\mathcal L=\tfrac12\mu(\eta)(\eta')^2+V(\eta)\ ,
\end{equation}
is conserved, i.e. $\tfrac{d}{d\xi}\mathcal{H}=0$. Integrating (\ref{eq:Vprime}), we obtain $2V=-\Psi/\delta+c$ (equivalently $\Psi'=2\mu R$)
with $\Psi$ as in (\ref{eq:fig}) and $c$ an integration constant. Thus, from (\ref{eq:beltrami}) we have
$$\mu(\eta)(\eta')^2-\Psi(\eta)/\delta=2\mathcal{H}-c\ ,$$ is constant. By setting $E= 2\mathcal{H}-c$, we obtain
(\ref{eq:fig}). The constant $E$ is determined by $\eta'\to0$ at the limiting equilibrium,
giving (\ref{eq:fi2g}).

Differentiating $\Psi$ in (\ref{eq:fig}) and multiplying with $(1+\eta)^2$, yields
$$
\Psi'(\eta)\,(1+\eta)^2=-2(1+\eta)^3+(2+\sigma C^2)(1+\eta)^2-C^2\ ,
$$
which is precisely the equilibrium cubic polynomial of (\ref{eq:cubiceq}). Its three roots in $w=1+\eta$ are $1+\eta_{+}$, $1+\eta_{-}$ and a third $w_2$. Since Vietta's rule implies that the product of the roots equals $-C^2/2<0$, while $1+\eta_{+},1+\eta_{-}>0$, the third root $w_2<0$, i.e.\ $\eta_2<-1$. Thus the only critical points of $\Psi$ with $\eta>-1$ are $\eta_{+}$ and $\eta_{-}$, and $\Psi$ is strictly monotone on $(\eta_{+},\eta_{-})$.
\end{proof}

We continue by studying the local behavior of the orbits at the singular point.

\begin{lemma}\label{lem:reggen}
The Lax condition $u_{+}+\sqrt{1+\eta_{+}}<s<u_{-}+\sqrt{1+\eta_{-}}$ is equivalent to $(1+\eta_{+})^3<C^2<(1+\eta_{-})^3$, i.e.\ to $\eta_{+}<\eta_{\ast}<\eta_{-}$. Let $\eta$ be either half-orbit of the connected orbit, with $\eta(\xi)\to\eta_{\ast}$ as $\xi\to0$. Then
\begin{equation}\label{eq:Phistarg}
\Phi_{\ast}:=\frac{\Psi(\eta_{\ast})-\Psi(\eta_{\mathrm e})}{\delta}\neq0\ ,
\end{equation}
and, as $\xi\to 0$,
\begin{equation}\label{eq:asympg}
|\eta(\xi)-\eta_{\ast}|\sim\left(\tfrac32\sqrt K\right)^{2/3}|\xi|^{2/3},\qquad
|\eta'(\xi)|\sim\left(\tfrac32\right)^{-1/3}K^{1/3}\,|\xi|^{-1/3}\ ,
\end{equation}
where $K=\tfrac{\sqrt[3]{C^2}}{3}\,|\Phi_{\ast}|>0$. 
Consequently $\eta,u\in C^{0,2/3}_{\mathrm{loc}}(\mathbb{R})$ up to $\xi=0$, while
$\eta',u'\in L^2_{\mathrm{loc}}(\mathbb{R})$ and $(u')^2\in L^1_{\mathrm{loc}}(\mathbb{R})$.
\end{lemma}

\begin{proof}
For the Lax equivalence, $C^2=(1+\eta_{+})^2(s-u_{+})^2$ gives $C^2>(1+\eta_{+})^3$, which is equivalent $(s-u_{+})^2>1+\eta_{+}$ and equivalently with $s-u_{+}>\sqrt{1+\eta_{+}}$ (as $s>u_{+}$). This is the lower Lax inequality. The upper one follows identically from $C^2=(1+\eta_{-})^2(s-u_{-})^2$.

By Lemma \ref{lem:firstintg}, $\Psi$ is strictly monotone on $(\eta_{+},\eta_{-})$, and $\eta_{\ast}$ lies strictly inside this interval, so $\Psi(\eta_{\ast})$ is strictly between $\Psi(\eta_{+})$ and $\Psi(\eta_{-})$. In particular
$\Psi(\eta_{\ast})\neq\Psi(\eta_{\mathrm e})$, which is (\ref{eq:Phistarg}). (If $\Phi_{\ast}=0$, both factors in (\ref{eq:fi2g}) would vanish at $\eta_{\ast}$ and $(\eta')^2$ would stay bounded, thus the orbit would cross $\eta_{\ast}$ smoothly.)

Since $(1+\eta_{\ast})^3=C^2$ and $\tfrac{d}{d\eta}(C^2-(1+\eta)^3)|_{\eta_{\ast}}=-3(1+\eta_{\ast})^2
=-3(C^2)^{2/3}$, the numerator $(1+\eta)^3(\Psi(\eta)-\Psi(\eta_{\mathrm e}))/\delta$ in (\ref{eq:fi2g}) tends to $C^2\Phi_{\ast}\neq0$ while $C^2-(1+\eta)^3=-3(C^2)^{2/3}(\eta-\eta_{\ast})(1+o(1))$. To see the last equality, set $w=1+\eta$ and $w_\ast=1+\eta_\ast$ so $w_\ast^3=C^2$. Then $C^2-w^3=w_\ast^3-w^3=(w_\ast-w)(w^2+w_\ast w+w_\ast^2)$. Observe that $w_\ast-w=\eta_\ast-\eta$ and the quadratic factor is continuous and tends to its value at $w_\ast$ as $\eta\to\eta_\ast$. Thus $w^2+ww_\ast+w_\ast^2\to 3w_\ast^2$ as $w\to w_\ast$. Combining these together, yields
$C^2-(1+\eta)^3=-(\eta-\eta_\ast)(3w_\ast^2+o(1))=-3 w_\ast(\eta-\eta_\ast)(1+o(1))$. Since $w_\ast^2=(C^2)^{2/3}$, we obtain
\begin{equation}\label{eq:helpful1}
C^2-(1+\eta)^3=-3(C^2)^{2/3}(\eta-\eta_\ast)(1+o(1)),\qquad \eta\to\eta_\ast\ .
\end{equation}
Hence we obtain
\begin{equation}\label{eq:localform}
(\eta')^2=\frac{C^2\Phi_{\ast}}{-3(C^2)^{2/3}(\eta-\eta_{\ast})}(1+o(1))
=\frac{K}{|\eta-\eta_{\ast}|}(1+o(1)),\qquad K=\tfrac13\sqrt[3]{C^2}\,|\Phi_{\ast}|>0\ ,
\end{equation}
where the sign is chosen so that $(\eta')^2>0$ on both sides. 
Since $(\eta')^2$ in \eqref{eq:fi2g} diverges as $\eta\to\eta_{\ast}$, we cannot integrate it with respect to $\xi$ in a straightforward way. Instead, we introduce the new variable $\Theta:=|\eta-\eta_{\ast}|^{3/2}$, chosen so that its $\xi$-derivative remains finite and nonzero at the singular point. In a neighborhood of $\xi=0$, the derivative $\eta'$ does not change sign, hence $|\eta-\eta_{\ast}|$ is monotone, and by the chain rule we obtain
$$
|\Theta'|=\tfrac{3}{2}|\eta-\eta_{\ast}|^{1/2}\,|\eta'|
=\tfrac32\sqrt{(\eta')^2\,|\eta-\eta_{\ast}|}\ .
$$
Placing the factor $|\eta-\eta_{\ast}|^{1/2}$ under the square root isolates exactly the product governed by the local behavior $(\eta')^2=K|\eta-\eta_{\ast}|^{-1}(1+o(1))$. The divergence compensates the vanishing factor, yielding
$$
(\eta')^2\,|\eta-\eta_{\ast}|=K\,(1+o(1))\longrightarrow K,
\qquad\text{and hence}\qquad
|\Theta'|\longrightarrow\tfrac32\sqrt K\quad(\xi\to0)\ .
$$
As $\Theta$ is continuous with $\Theta(0)=0$ and $\Theta\ge0$, and $\Theta'$ tends to the constant $\tfrac32\sqrt{K}$, the mean value
theorem gives $\Theta(\xi)/|\xi|\to\tfrac32\sqrt K$, i.e.
$$
\Theta(\xi)=\tfrac32\sqrt{K}\,|\xi|\,(1+o(1))\ .
$$
Raising to the power $2/3$ and recalling $\Theta=|\eta-\eta_{\ast}|^{3/2}$ yields the
first relation in \eqref{eq:asympg},
$$
|\eta(\xi)-\eta_{\ast}|=\left(\Theta(\xi)\right)^{2/3}
=\left(\tfrac32\sqrt K\right)^{2/3}|\xi|^{2/3}(1+o(1))\ .
$$
For the second relation, we take the square root of the local form (\ref{eq:localform}) to obtain
$|\eta'|=\sqrt{K}\,|\eta-\eta_{\ast}|^{-1/2}(1+o(1))$ and substitute the first relation. Since $\sqrt K\left(\tfrac32\sqrt K\right)^{-1/3}=\left(\tfrac32\right)^{-1/3}K^{1/3}$,
$$
|\eta'(\xi)|=\sqrt K\,\left(\tfrac32\sqrt K\right)^{-1/3}|\xi|^{-1/3}(1+o(1))
=\left(\tfrac32\right)^{-1/3}K^{1/3}\,|\xi|^{-1/3}(1+o(1))\ ,
$$
which is also what one obtains by differentiating the first relation directly.

Hence the first relation in (\ref{eq:asympg}), and the second follows from $|\eta'|\sim\sqrt K\,|\eta-\eta_{\ast}|^{-1/2}$. Since $u=s+C/(1+\eta)$ is smooth in $\eta$ away from $\eta=-1$, $u$ inherits the exponent $2/3$ and $u'=-C\eta'/(1+\eta)^2$ satisfies $|u'|\sim\mathrm{const}\,|\xi|^{-1/3}$, as $2/3<1$, $\eta',u'\in L^2_{\mathrm{loc}}$ and $(u')^2\sim|\xi|^{-2/3}\in L^1_{\mathrm{loc}}$.
\end{proof}

The two half-orbits each solve the profile equation classically away from the singular point, but $\eta'$ and $u'$ blow up like $|\xi|^{-1/3}$ as $\xi\to0$, and the momentum balance contains the product $(u-s)u''$, which is ill-defined at the singularity. To glue the half-orbits into a single distributional solution one must therefore control the momentum flux across $\xi=0$: if the flux was discontinuous, its derivative would carry a Dirac mass $[\,\cdot\,]_0\,\delta_0$ at the front and the glued profile would solve the equation only up to a spurious point source. The next lemma shows that no such jump occurs. (Throughout we write $[f]_0:=f(0^{+})-f(0^{-})$ for the jump of a piecewise-smooth function across the gluing point $\xi=0$). Working with the flux variable $P=\eta+\tfrac12(u-s)^2$, it establishes that $P$ extends across the singular point as a $C^1$ function with $P'(0)=0$, the blow-up of $\eta'$ being exactly compensated by the simple zero of $(1+\eta)^3-C^2$ at $\eta_{\ast}$. This continuity of the flux is the
no-jump condition that promotes the two classical half-orbits to a weak solution of the momentum equation.

\begin{lemma}\label{lem:fluxg}
Let $\eta$ be the glued profile of Lemma~\ref{lem:reggen} and $P=\eta+\tfrac12(u-s)^2$ with $u=s+C/(1+\eta)$. Then $P\in C^1(\mathbb{R})$ with $P'(0)=0$ and $P'(\xi)=O(|\xi|^{1/3})$ as $\xi\to0$, and $P''\in L^1_{\mathrm{loc}}(\mathbb{R})$ with no singular part.
\end{lemma}

\begin{proof}
Eliminating $u$ through the algebraic relation $u=s+C/(1+\eta)$ of (\ref{eq:systemd3}) expresses the flux $P$ is written as
$$
P=\eta+\frac{C^2}{2(1+\eta)^2}\ .
$$
On $\mathbb{R}\setminus\{0\}$ the profile $\eta$ is smooth and $1+\eta>0$, so $P$ is smooth there, since $\eta$ is continuous on $\mathbb{R}$ and $1+\eta_{\ast}=\sqrt[3]{C^2}>0$, $P$ is continuous on all of $\mathbb{R}$, with $P(0)=\eta_{\ast}+\tfrac12\sqrt[3]{C^2}$.

By the chain rule, for $\xi\neq0$,
$$
P'=\frac{(1+\eta)^3-C^2}{(1+\eta)^3}\,\eta'\ .
$$
As $\xi\to0$ one has $(1+\eta)^3\to C^2\neq0$, while by (\ref{eq:asympg}) and the
linearization of the cubic at its simple root $\eta_{\ast}$,
$$
(1+\eta)^3-C^2\sim 3(C^2)^{2/3}(\eta-\eta_{\ast})\sim \mathrm{const}\,|\xi|^{2/3},
\qquad
\eta'\sim \mathrm{const}\,|\xi|^{-1/3}\ .
$$
Hence the blow-up of $\eta'$ is overcompensated by the vanishing of the numerator and
$$
P'(\xi)=O(|\xi|^{1/3})\longrightarrow 0\qquad(\xi\to 0)\ .
$$
In fact the leading coefficients do not cancel, since $\Phi_{\ast}\neq0$ (equivalently $K\neq0$). A direct computation gives $P'(\xi)\sim c\,|\xi|^{1/3}$ with $c\neq 0$, so the rate $1/3$ is sharp and cannot be improved. Because $P$ is continuous on $\mathbb{R}$, is $C^1$ on $\mathbb{R}\setminus\{0\}$, and $P'$ admits the limit
$0$ as $\xi\to0$ from both sides, the mean value theorem gives differentiability at the origin with $P'(0)=0$ and continuity of $P'$ there, thus $P\in C^1(\mathbb{R})$.

For the second derivative we use the once-integrated balance (\ref{eq:starfluxg}), which on each half-line holds classically and, solved for $P''$, reads
$$
P''=\frac{1}{\delta}\,(P-G)+\frac32\,(u')^2,\qquad (u')^2=\frac{C^2}{(1+\eta)^4}\,(\eta')^2 .
$$
The first term is continuous, hence locally bounded. For the second, $(1+\eta)^{-4}\to (C^2)^{-4/3}$ as $\xi\to0$ and, by (\ref{eq:asympg}), $(\eta')^2\sim \mathrm{const}\,|\xi|^{-2/3}$, so $(u')^2\sim \mathrm{const}\,|\xi|^{-2/3}$, which is integrable near $0$ as $-2/3>-1$. Therefore the right-hand side belongs to $L^1_{\mathrm{loc}}(\mathbb{R})$ and is a function carrying no atom at $\xi=0$. Since $P'$ is continuous across the origin (no jump), its distributional derivative coincides with this $L^1_{\mathrm{loc}}$ function: $P''\in L^1_{\mathrm{loc}}(\mathbb{R})$ with no singular part. This completes the proof.
\end{proof}

We are now ready to introduce the formal definition of a weakly singular shock wave for system (\ref{eq:bous}):

\begin{definition}\label{def:weakg}
Let $s>1$ and let $\eta_{\pm},u_{\pm}$ satisfy the shock relations (\ref{eq:equilib}) and the Lax condition. A pair $(\eta,u)\in C(\mathbb{R})^2$ is a \emph{weakly singular traveling-wave solution} of (\ref{eq:bous}) with speed $s$ and asymptotic states $\eta_{\pm},u_{\pm}$ if
\begin{itemize}
\item[(i)] $\eta\to \eta_{\pm}$ and $u\to u_{\pm}$ as $\xi\to\pm\infty$,
\item[(ii)] $\eta',u'\in L^2_{\mathrm{loc}}(\mathbb{R})$ (so $(u')^2\in L^1_{\mathrm{loc}}(\mathbb{R})$),
\item[(iii)] with $P$ given by (\ref{eq:Pdefg}), the identities
\begin{equation}\label{eq:weakprofileg}
-s\eta'+\big[(1+\eta)u\big]'=0\qquad\text{and}\qquad
\Big[P-\delta P''+\tfrac32\delta(u')^2\Big]'=0\ ,
\end{equation}
hold in $\mathcal{D}'(\mathbb{R})$.
\end{itemize}
Equivalently, $(\eta(x-st),u(x-st))$ is a solution of (\ref{eq:bous}) in $\mathcal{D}'(\mathbb{R}^2)$.
\end{definition}

Hypothesis (ii) guarantees $(u')^2\in L^1_{\mathrm{loc}}$, so that the second identity involves only the distributional derivative of $P''\in \mathcal{D}'(\mathbb{R})$ and the regular distribution $(u')^2$. The first identity is the mass equation and is equivalent, by continuity, to the algebraic first integral $u=s+\tfrac{C}{1+\eta}$ with $C=(1+\eta_{+})(u_{+}-s)$, while the second is the flux form of the momentum equation. We next show that the glued solutions are indeed solutions of the regularized shallow water system in the sense of distributions.

\begin{theorem}\label{thm:weakg}
Let $\eta_{\pm},u_{\pm}$ satisfy (\ref{eq:equilib}) and the Lax condition, with $\eta_{-}>\eta_{+}$. The phase-plane construction of Section \ref{sec:existence}, applied to (\ref{eq:systemd4}), yields two half-orbits, one
limiting on $(\eta_{-},0)$ as $\xi\to-\infty$, one on $(\eta_{+},0)$ as $\xi\to+\infty$, both approaching $\eta_{\ast}$ as $\xi\to0$. Let $\eta$ be the glued solution at $\xi=0$ and $u=s+\tfrac{C}{1+\eta}$. Then $(\eta,u)$ is a weakly singular traveling-wave solution in the sense of Definition \ref{def:weakg}.
\end{theorem}

\begin{proof}
Items (i)--(ii) follow from the construction and Lemma~\ref{lem:reggen}. The two half-orbits attain the asymptotic states $\eta_{\pm}$ as $\xi\to\pm\infty$ and the common value $\eta_{\ast}$ as $\xi\to0$. Since $\eta_{+}<\eta_{\ast}<\eta_1$, Lemma \ref{lem:reggen} produces the same limit $\eta_{\ast}$ from either side, so the glued profile $\eta$ is continuous across $\xi=0$, and the same lemma gives $\eta',u'\in L^2_{\mathrm{loc}}(\mathbb{R})$. 

Since $\eta$ and $u$ are continuous at $\xi=0$, the once-integrated mass relation $-s\eta+(1+\eta)u=A$ of (\ref{eq:systemd2}), with $A=s+C$ as in (\ref{eq:k1}), holds across the origin and hence on all of $\mathbb{R}$. Equivalently $u=s+C/(1+\eta)$, the first equation of (\ref{eq:systemd3}). Both $\eta$ and $(1+\eta)u$ are continuous across $\xi=0$, so $[(1+\eta)u]_0=0$, and differentiating the relation in $\mathcal{D}'(\mathbb{R})$ produces no Dirac mass at the origin:
$$
-s\eta'+\left[(1+\eta)u\right]'=0\qquad\text{in }\mathcal{D}'(\mathbb{R})\ ,
$$
which is the first identity of (\ref{eq:weakprofileg}).

On $\mathbb{R}\setminus\{0\}$ each half-orbit solves (\ref{eq:systemd3}) classically, so the once-integrated balance
(\ref{eq:starfluxg}) holds pointwise almost everywhere on $\mathbb{R}$. Set
$$
\Pi:=P-\delta P''+\tfrac32\delta(u')^2\ ,
$$
so that $\Pi=G$ a.e. By Lemma~\ref{lem:fluxg}, $P\in C^1(\mathbb{R})$ and $P''\in L^1_{\mathrm{loc}}(\mathbb{R})$ with no singular part, and since $(u')^2\in L^1_{\mathrm{loc}}(\mathbb{R})$ as well, $\Pi\in L^1_{\mathrm{loc}}$. Two $L^1_{\mathrm{loc}}$ functions that agree almost everywhere define the same distribution, so pairing with an arbitrary $\phi\in\mathcal{D}(\mathbb{R})$ promotes the almost-everywhere equality to an identity
in $\mathcal{D}'(\mathbb{R})$,
$$
\left\langle\,\Pi-G,\ \phi\,\right\rangle
=\int_{\mathbb{R}}\left(P-\delta P''+\tfrac32\delta(u')^2-G\right)\,\phi~d\xi=0\ ,
$$
that is, $\Pi\equiv G$ in $\mathcal{D}'(\mathbb{R})$ and not merely on each half-line separately. Differentiating this distributional identity gives
$$
\Pi'=0\qquad\text{in }\mathcal{D}'(\mathbb{R})\ ,
$$
the second identity of (\ref{eq:weakprofileg}). Note that no boundary term $[\Pi]_0\,\delta_0$ arises at the origin, consistently with the continuity of $\Pi$ there ($[\Pi]_0=0$ by Lemma~\ref{lem:fluxg}). 

Together with the mass equation, (\ref{eq:weakprofileg}) is exactly (\ref{eq:bous}) under the traveling wave {\em ansatz}, and the proof is complete.
\end{proof}

\begin{remark}\label{rem:canon}
The cusped soliton of homoclinic-type orbits (both half-orbits limiting on the same equilibrium $\eta_{+}$) is glued at the same $\eta_{\ast}$ and is, by similar arguments, a weak solution. The rigorous theory does not distinguish it from the shock, so the failure to observe cuspons in numerical simulations (see Section \ref{sec:numericalexp}) is a matter of dynamical stability, not existence.
\end{remark}

\section{Energy approximation}\label{sec:conservation}

In this section we show that solutions of system (\ref{eq:bous2}) have energy that approximates the corresponding energy of system (\ref{eq:sw1}) as $\delta\to 0$.  

Let $(\eta,u)$ be a solution of (\ref{eq:bous2}) with initial conditions $\eta(x,0)=H(x)$ and $u(x,0)=U(x)$. We observe, that the second equation of system (\ref{eq:bous2}), after rearranging its terms, can be written as
\begin{equation}
\mathcal{D}_t\, L\, u+gL\,\eta_x=0\ ,
\end{equation}
where $L=I-\delta \partial_x^2$ is a BBM-type smoothing operator, and $\mathcal{D}_t =\partial_t+u\partial_x$ is the usual material derivative operator.

For the operator product $\mathcal{D}_t L$, it can be verified that 
\begin{equation}
\mathcal{D}_t Lv=L \mathcal{D}_t v+\delta[\mathcal{D}_t,-\partial_x^2]v\ ,
\end{equation}
where $[A,B]=AB-BA$ is the commutator operator. Furthermore, it holds true that
\begin{equation}
[\mathcal{D}_t,-\partial_x^2]u=3u_xu_{xx}=\frac{3}{2}\left(u^2_x\right)_x\ .
\end{equation}
Thus, system (\ref{eq:bous}) is written in the following form as a perturbation of the classical shallow water equations:
\begin{equation}\label{eq:systemd5}
\begin{aligned}
&\eta_t+[(D+\eta)u]_x=0\ ,\\
&u_t+g\eta_x+\tfrac{1}{2}(u^2)_x + \frac{3}{2}\delta \left(L^{-1} u_x^2\right)_x=0\ .
\end{aligned}
\end{equation}
If we multiply the first equation of (\ref{eq:systemd5}) by $\eta + \tfrac{1}{2}u^2$, the second by $(D + \eta)u$, we obtain the energy balance equation
\begin{equation}\label{eq:energybal}
\mathcal{E}_t+\mathcal{L}_x=-\frac{3}{2}\delta \left(L^{-1} u_x^2\right)_x(D+\eta)u\ ,
\end{equation}
where 
$$\mathcal{E}=g\eta^2+ (D+\eta) u^2\quad \text{and}\quad \mathcal{L}=\tfrac{1}{2}(D+\eta)u^3+g(D+\eta)\eta u\ .$$

For smooth solutions that decay at infinity with an appropriate rate, we multiply the first equation of (\ref{eq:systemd5}) by $\eta + \tfrac{1}{2}u^2$, the second by $(D + \eta)u$, and integrate over $\mathbb{R}$. This yields
\begin{equation}\label{eq:energy1}
\frac{d}{dt}E(t) = -\frac{3}{2}\delta \int_{-\infty}^{+\infty} (\partial_x L^{-1} u_x^2)(D+\eta)u~dx\ ,
\end{equation}
where
$$E(t)=\tfrac{1}{2}\int_{-\infty}^{+\infty}[ g\eta^2+ (D+\eta) u^2]~dx\ ,$$
is the functional of the total energy of (\ref{eq:sw1}). Therefore, integration in time yields
\begin{equation}\label{eq:energyd1}
E(t)=E(0)-\frac{3}{2}\delta \int_0^t \int_{-\infty}^{+\infty} (\partial_x L^{-1} u_x^2)(D+\eta)u~dx~d\tau\ .
\end{equation}

If $(\tilde{\eta},\tilde{u})$ is a solution of (\ref{eq:sw1}) with the same initial conditions as before, i.e. $\tilde{\eta}(x,0)=H(x)$ and $\tilde{u}(x,0)=U(x)$, then the energy of this solution will be preserved as
\begin{equation}\label{eq:energyd2}\tilde{E}(t)=\int_{-\infty}^{+\infty}[ g\tilde{\eta}^2+ (D+\tilde{\eta}) \tilde{u}^2]~dx= E(0)\ .
\end{equation}

Before estimating the difference between the two energies, we provide an estimate for the operator $L^{-1}$:
\begin{lemma}\label{lem:kernel}
The operator $L^{-1}=(I-\delta\partial_x^2)^{-1}$ on $\R$ is convolution with the Green's function $G_\delta(x)=\tfrac{1}{2\sqrt\delta}\,e^{-|x|/\sqrt\delta}$, so that $\partial_x L^{-1}f=G_\delta'*f$ with
$$
G_\delta'(x)=-\frac{1}{2\delta}\,\mathrm{sgn}(x)\,e^{-|x|/\sqrt\delta}\ ,
\qquad
\|G_\delta'\|_{L^1}=\frac{1}{\sqrt\delta}\ .
$$
In particular $\partial_x L^{-1}$ is bounded on $L^1(\R)$, with $\|\partial_x L^{-1}f\|_{L^1}\le\delta^{-1/2}\,\|f\|_{L^1}$.
\end{lemma}

\begin{proof}
The symbol of $L^{-1}$ is $\widehat{G_\delta}(k)=(1+\delta k^2)^{-1}$, whose inverse Fourier transform is $G_\delta$ (and $\int_\R G_\delta=1$). Differentiating gives $G_\delta'$ as stated, and $\|G_\delta'\|_{L^1}=\tfrac{1}{2\delta}\int_\R e^{-|x|/\sqrt\delta}~dx = \tfrac{1}{2\delta}\cdot2\sqrt\delta=\delta^{-1/2}$. Young's inequality for convolutions yields $\|G_\delta'*f\|_{L^1}\le\|G_\delta'\|_{L^1}\|f\|_{L^1}$, and the mapping bound follows.
\end{proof}

Now we can estimate the difference of the two energies, and precisely, that $|E(t)-\tilde{E}(t)| = O(\sqrt{\delta}\, t)$.

\begin{proposition}\label{prop:energy}
Let $(\eta,u)$ solve (\ref{eq:bous2}) on $[0,T]$ with $u_x(\cdot,t)\in L^2(\R)$ and $(D+\eta)(\cdot,t)\,u(\cdot,t)\in L^\infty(\R)$ for each $t\in[0,T]$, and let $(\tilde\eta,\tilde u)$ solve (\ref{eq:sw1}) with the same initial data. Then, for all $t\in[0,T]$,
$$
|E(t)-\tilde E(t)|\ \le\ \frac32\sqrt\delta\int_0^t
\|u_x(\tau)\|_{L^2}^2\,\big\|(D+\eta)(\tau)\,u(\tau)\big\|_{L^\infty}\,d\tau\ .
$$
\end{proposition}

\begin{proof}
Subtracting (\ref{eq:energyd2}) from (\ref{eq:energyd1}),
$$
E(t)-\tilde E(t)=-\frac32\delta\int_0^t\!\!\int_\R
\left(\partial_x L^{-1}u_x^2\right)(D+\eta)u\,dx\,d\tau\ .
$$
By H\"older's inequality in the $L^1$--$L^\infty$ pairing and then Lemma
\ref{lem:kernel},
$$
\Big|\int_\R\left(\partial_x L^{-1}u_x^2\right)(D+\eta)u\,dx\Big|
\le\big\|\partial_x L^{-1}u_x^2\big\|_{L^1}\,\|(D+\eta)u\|_{L^\infty}
\le\delta^{-1/2}\,\|u_x^2\|_{L^1}\,\|(D+\eta)u\|_{L^\infty}\ .
$$
Since $\|u_x^2\|_{L^1}=\|u_x\|_{L^2}^2$, taking absolute values inside the time
integral gives the stated bound.
\end{proof}

The bound of Proposition \ref{prop:energy} is finite for weakly singular solutions because $u_x\sim|\xi|^{-1/3}$ gives
$u_x^2\sim|\xi|^{-2/3}\in L^1$, while $\|(D+\eta)u\|_{L^\infty}<\infty$ because $\eta$
and $u$ are bounded. 
Furthermore, numerical experiments presented in Section \ref{sec:numericalexp} confirm that energy dissipation occurs in the presence of weakly singular fronts, in a manner consistent with the nonlinear shallow water equations (\ref{eq:sw1}), while overall the regularized system retains an energy comparable to that of the corresponding shallow water solutions.

\section{Finite volume methods}\label{sec:finitevols}

In this section we generalize the finite volume methods introduced in \cite{dkm2011,dkm2013,DCMM2013} to the high-order system (\ref{eq:bous}). For notation convenience, we write system (\ref{eq:bous}) in the conservative form:
\begin{equation}\label{E2.2}
(I-\delta\, \partial_x^2){\bw}_t+\left[{\bF}({\bw})\right]_x-\delta\,\left[{\bG}({\bw})\right]_x=0, 
\end{equation}
where ${\bw}=(\eta,u)^T$ and
$${\bF}({\bw})=\begin{pmatrix}(D+\eta)u\\
\eta+\frac{1}{2}u^2\end{pmatrix},\quad {\bG}(\bw)=\begin{pmatrix} 0\\
\eta_{xx}+uu_{xx}-\tfrac{1}{2}u_x^2
\end{pmatrix}\ ,$$
are the hyperbolic and dispersive flux functions, respectively.

For the discretization of the domain, we consider a uniform partition $\mathcal{T}= \{x_i, \ i\in\mathbb{Z} \}$ of $\mathbb{R}$ into cells $C_i= (\xim,\xip)$ where $x_i = (x_{\ip}+x_{\im})/2$. Let $\Delta x_i= \xip-\xim$ be the length of  the cell $C_i$, and  $\Delta x_{\ip}=x_{i+1}-x_i$, where $\Delta x_i=\Delta x_{\ip}=\Delta x$ for all $i\in\mathbb{Z}$. Although here we consider uniform grids, the method can be extended to nonuniform or even moving grids for better resolution of the singular wave front \cite{KD2017,KDMS2019}.

We approximate the solution $\bw$ using a piecewise constant function $\bW(x,t)=\sum_{i\in\mathbb{Z}} \bW_i(t)\chi_{C_i}(x)$, where $\chi_{C_i}$ is the characteristic function of the cell $C_i$, and $\bW_i$ is the cell-average approximation of $\bW$ on $C_i$. 
Thus, the vector $\bW$ comprises the cell-average approximations of $\eta$ and $u$. We denote its entries as $\bW_i=(\bar{\eta}_i,\bar{u}_i)^T$.

The semidiscrete approximation of $(\ref{E2.2})$ over the cell $C_i$ is defined as
\begin{equation}
(I-\delta \bD^2)\bW_t+\frac{1}{\Delta x}\left[\bF(\bW(x_{\ip},t))-\bF(\bW(x_{\im},t)) \right]-\frac{\delta}{\Delta x}\left[\bG(\bW(x_{\ip},t))-\bG(\bW(x_{\im},t))\right]=0\ ,
\end{equation}
where $\bD^2$ is the second order finite difference operator of the second derivative defined as 
$$\bD^2_i\bar{w}=\frac{\bar{w}_{i+1}-2\bar{w}_i+\bar{w}_{i-1}}{\Delta x^2}\ .$$

Due to the discontinuity of the discrete solution $\bW$ at the cell interfaces $x_{\ip}$, $i\in\mathbb{Z}$, we approximate the flux at the cell interfaces by a numerical flux function 
$$\bF(\bW(x_{\ip},t))\approx \mathcal{F}_{\ip}(\bW^L_{\ip},\bW^R_{\ip})\ ,$$
where $\bW^{L,R}_{\ip}$ are reconstructions of the variable $\bW$ to the left and right sides of each cell interface, respectively. Such reconstructions are described later in this section.

In order to define the numerical flux function as introduce the Jacobian matrix of the flux function $\bF$, which is 
$$
\bA=\begin{pmatrix}
u & D+\eta\\ 
1 & u 
\end{pmatrix}\ ,
$$ 
with eigenvalues  $\lambda_i=u\pm \sqrt{1+\eta}$, $i=1,2$. 

Then, we  approximate the flux function $\bF$ by one of the following numerical flux functions:
\begin{align}
& \mathcal{F}^{KT}(\bw,\bv) = \frac{1}{2}\left\{\left[ \bF(\bv) + \bF(\bw)\right] - \mathcal{A}(\bw,\bv)\left[\bv-\bw\right]\right\} \label{KTFlux} \\
& \mathcal{F}^{CF}(\bw,\bv)  = \frac{1}{2}\left\{\left[ \bF(\bv) + \bF(\bw)\right] - \mathcal{B}(\bw,\bv)\left[\bF(\bv)-\bF(\bw)\right]\right\}. \label{CFFlux}
\end{align}

The flux \eqref{KTFlux} is a Lax-Friedrichs type flux, and the operator $\mathcal{A}$ is defined as 
\begin{equation}\label{KTA}
\mathcal{A}(\bw,\bv) = \max\left[\rho\left(\bA(\bw)\right), \rho\left(\bA(\bv)\right)\right], 
\end{equation}
where $\rho(\bA)$ is the spectral radius of the Jacobian matrix $\bA$, \cite{KT, NT}. 

For the \emph{characteristic} flux function (\ref{CFFlux}), the operator $\mathcal{B}$ is
\begin{equation}\label{CFA}
\mathcal{B}(\bw,\bv) = \sign\left(\bA\left(\frac{\bw+\bv}{2}\right)\right)\ ,
\end{equation}
and is similar to the upwind flux, \cite{Ghidaglia1996, Ghidaglia2001}.

For the approximation of dispersive flux $\bG(\bW(x_{\ip},t))$ we extend the average fluxes introduced in \cite{dkm2011}, to incorporate the new high-order terms. Specifically, we consider the average flux:
\begin{equation}\label{eq:dispt1}
\mathcal{G}_{\ip}(\bW)=\frac{\bD^2_{i+1}\bar{\eta}+\bD^2_{i}\bar{\eta}}{2}+\left(\frac{\bD^1_{i+1}\bar{u}+\bD^1_i\bar{u}}{2}\right)^2- \bar{u}_{i+1}\frac{\bD^2_{i+1}\bar{u}+\bD^2_{i}\bar{u}}{2}\ ,
\end{equation}
where
$$\bD^2_{i}\bar{w}=\frac{\bar{w}_{i+1}-2\bar{w}_i+\bar{w}_{i-1}}{\Delta x^2}\quad 
\text{and}\quad \bD^1_i\bar{w}=\frac{\bar{w}_{i+1}-\bar{w}_i}{\Delta x}\ .$$

Employing the numerical fluxes we can write the discrete system as a system of equations where for each $i\in \mathbf{Z}$:
\begin{equation}\label{eq:FV1}
\begin{aligned}
(I-\delta \bD^2_i)\frac{d}{dt}\bW_i+&\frac{1}{\Delta x}\left[\mathcal{F}_{\ip}(\bW_{\ip}^L,\bW_{\ip}^R)-\mathcal{F}_{\im}(\bW_{\im}^L,\bW_{\im}^R) \right]\\
&-\frac{\delta}{\Delta x}\left[\mathcal{G}_{\ip}(\bW)-\mathcal{G}_{\im}(\bW)\right]=0\ ,
\end{aligned}
\end{equation}
where $\bW_i(0)=(\eta(x_i,0),u(x_i,0))$. 

In practice the system must be posed on a finite interval $(-\ell,\ell)$ rather than the whole real line, which requires the imposition of boundary conditions such as reflective or Dirichlet conditions \cite{MSM2017}. As no well-posedness theory is available for the present model on bounded domains, we work primarily with periodic boundary conditions. For the weakly singular shock waves, however, we additionally employ Dirichlet boundary
conditions, in order to verify the convergence of the method and to compare with the known convergence rates for discontinuous shock waves. 

To achieve second order accuracy in space, the cell averaged solution is approximated by a piecewise polynomial function. The core strategy here involves constructing second-order approximations $\bW_{\ip}^L,\bW_{\ip}^R$ to $\bW(\xip,t)$ based on the computed cell averages $\bW_i$.
The most straightforward approach for assigning the values $\bW_{\ip}^L,\bW_{\ip}^R$ is the piecewise constant approximation within each cell, as given by:
\begin{equation}
\bW_{\ip}^L = \bW_i, \quad \bW_{\ip}^R = \bW_{i+1},
\end{equation}
which leads to a semidiscrete finite volume scheme that is formally first-order accurate in space. To enhance accuracy, we apply classical MUSCL-type (TVD2) linear reconstruction \cite{Leer1979,Sweby1984}.

Common choices for the limiter function include the MinMod, Van Leer, Monotonized, and Van Albada limiters. These limiters, with the exception of MinMod, generally provide sharper resolution of discontinuities by avoiding excessive damping of slopes near steep gradients. The TVD2 method achieves second-order accuracy, except at local extrema, where it reduces to first-order accuracy. As an alternative, we consider the UNO2 reconstruction method \cite{HaOs}, which results in second-order accuracy even at local extrema. 

To discretize the time domain, we begin with an initial condition $\bW(t^0)$. Subsequently, we invert the discrete elliptic operator $I - \delta \bD^2$ to obtain a system of ordinary differential equations. The discrete elliptic operator with periodic boundary conditions can be inverted efficiently by means of Fast Fourier Transform techniques, while for non-periodic boundary conditions one can use the SMW iteration proposed in \cite{M2024}. Inverting the regularization operator leads to a system that is not stiff. To solve this system, we employ the optimal third-order Strong Stability Preserving (SSP) Runge–Kutta method, $\text{SSPRK}(3,3)$, as specified in \cite{GKS2011} with a uniform time step. In this context, we denote the uniform time step as $\Delta t=t^{j+1}-t^j$. Since the resulting system of ordinary differential equations is not stiff, we do not enforce stringent constraints on $\Delta t$. Nonetheless, we typically set $\Delta t = \Delta x/10$ (unless otherwise is stated) to reduce errors introduced by the time discretization. For further details, please refer to \cite{dkm2011}. 

\section{Numerical experiments}\label{sec:numericalexp}

In this section we explore the approximation of weakly singular shock wave solutions using the aforementioned finite volume methods. We also compare the corresponding numerical solutions of the shallow water equations.

\subsection{Generation of weakly singular shock waves}\label{sec:wekshoK}

To generate weakly singular shock waves as those in Figure \ref{fig:rshock0}(a) we consider the initial condition
\begin{equation}\label{eq:initconddb}
\eta(x,0)=\frac{\tanh(\kappa(x+\zeta))-\tanh(\kappa(x-\zeta))}{4}, \qquad u(x,0)=0 \ ,
\end{equation}
in the interval $[-\ell,\ell]$ with periodic boundary conditions. Such initial conditions are commonly used to describe the dam-break problem in the domain $[0,\ell]$. For this experiment we take $\kappa=0.1$, $\zeta=250$ and $\ell=1000$. We integrate system (\ref{eq:bous1}) using the fully discrete scheme of the previous section for $t\in (0,T]$ with $T=600$ and with $\Delta x=0.05$, $\Delta t=\Delta x/2$.

For the initial condition (\ref{eq:initconddb}), two symmetric weakly singular shocks and two rarefaction waves are generated, traveling in opposite directions. The rarefaction waves collide and cancel each other, leaving behind two weakly singular waves that connect the two different states: $\eta_0=\eta_{+}=0$ and $\eta_1=\eta_{-}\approx 0.237549$. While $u_{+}=0$, $u_{-}=0.224585$.  
The speed of this weakly singular shock is $s\approx 1.17$. This can be found by solving the relations (\ref{eq:criticals}) with respect $s$  for a critical point $\eta_1$. Specifically, its formula is 
$$s^2=2\frac{1 + 2 \eta_1 + \eta_1^2}{2 + \eta_1}\ .$$

\begin{figure}[t]
\centering
\includegraphics[width=\columnwidth]{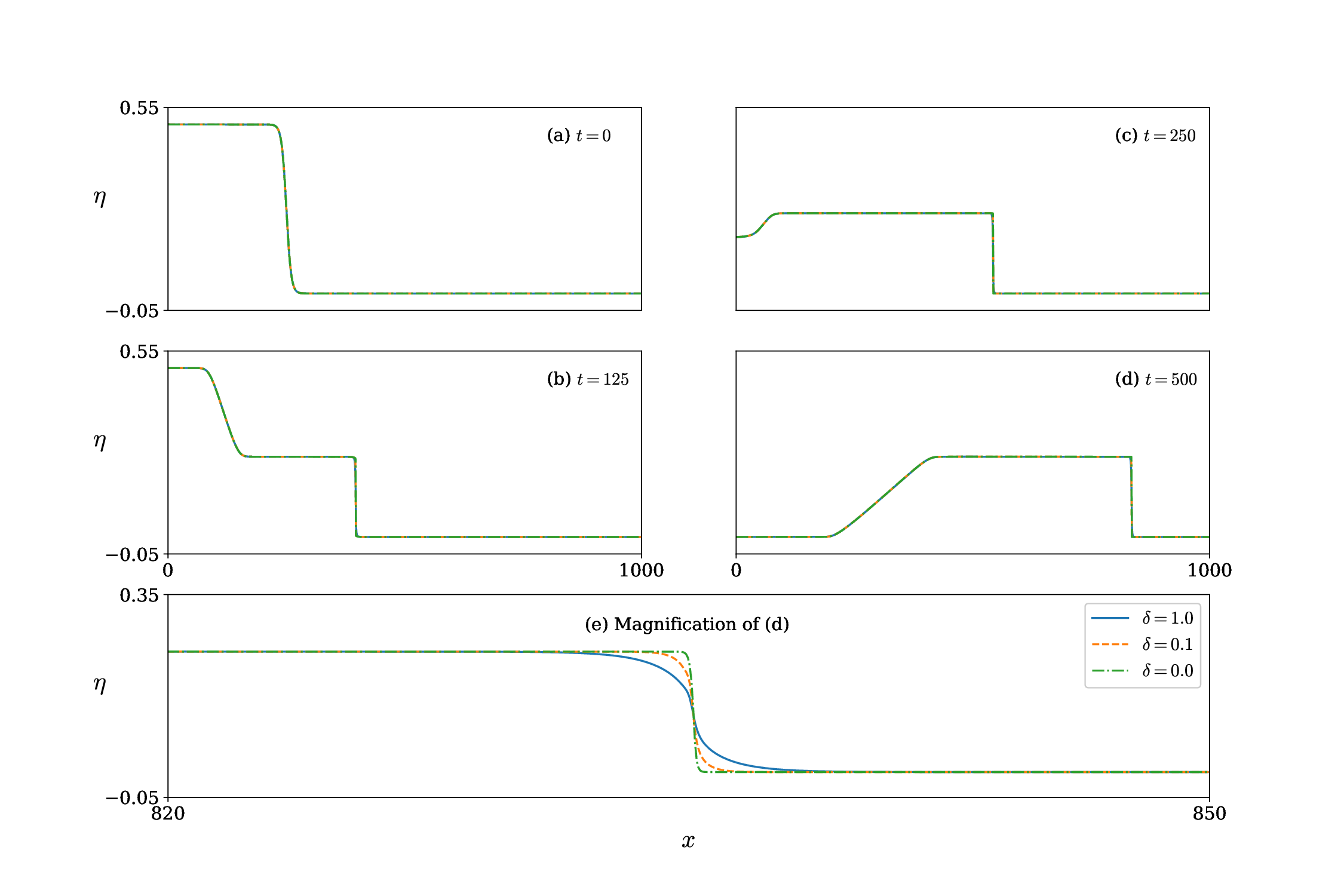}
\caption{Generation of weakly singular shock waves for $\delta=1$, $0.1$ and comparison with classical shock wave of $\delta=0$ using the Kurganov-Tadmor numerical flux with the MinMod limiter with the standard TVD2 reconstruction}
\label{fig:evolv1211}
\end{figure}

Figure \ref{fig:evolv1211} presents the numerical results for times $t=0,125,250$ and $500$ obtained using the Kurganov-Tadmor numerical flux with the MinMod limiter with the standard TVD2 reconstruction for $\delta=1$, $0.1$ and $0$. 

\begin{figure}[t]
\centering
\includegraphics[width=\columnwidth]{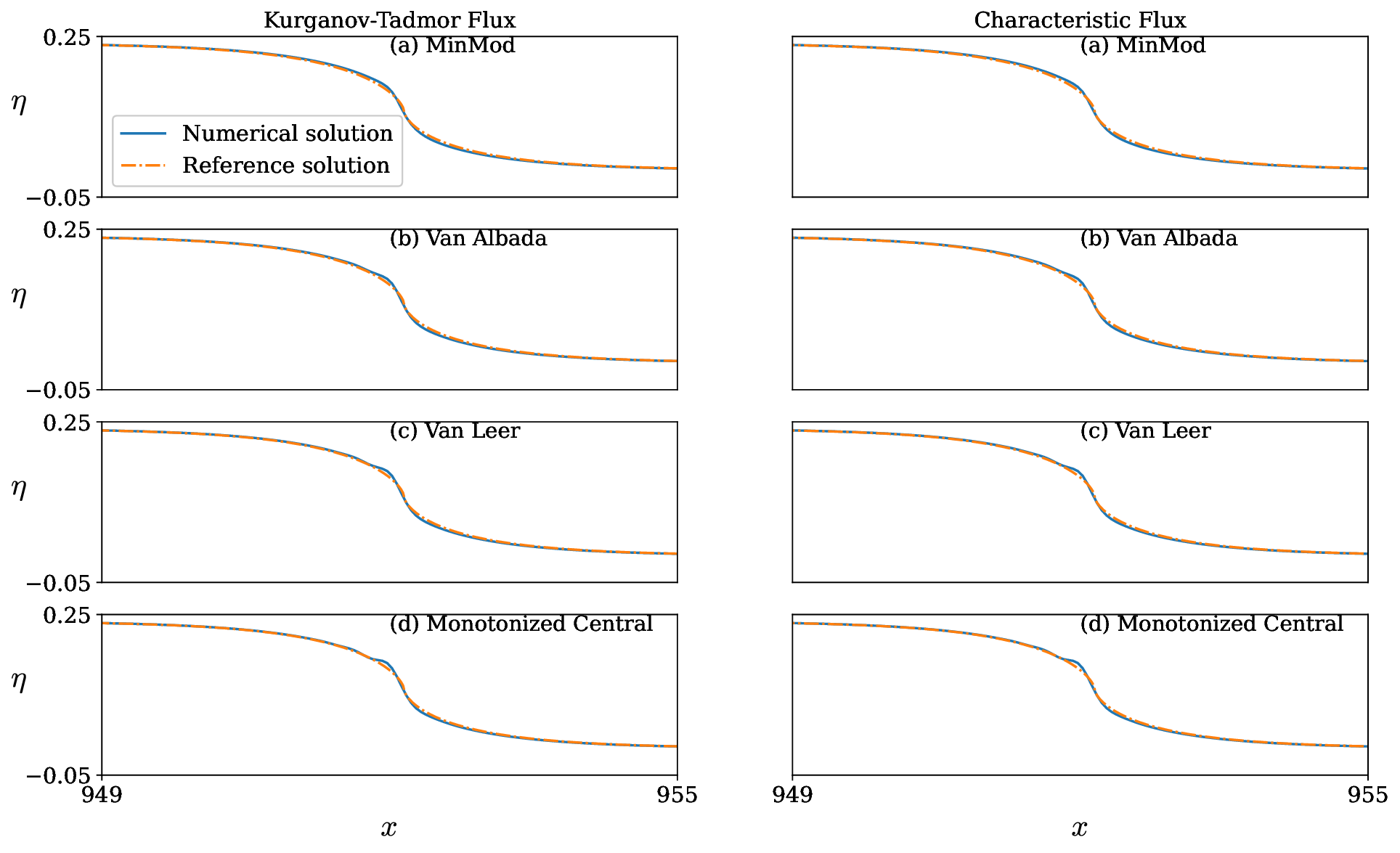}
\caption{Comparison between various limiters and numerical fluxes with TVD2 reconstruction for the approximation of the weakly singular front with $\delta=1$}
\label{fig:compfl}
\end{figure}

\begin{figure}[t]
\centering
\includegraphics[width=0.65\columnwidth]{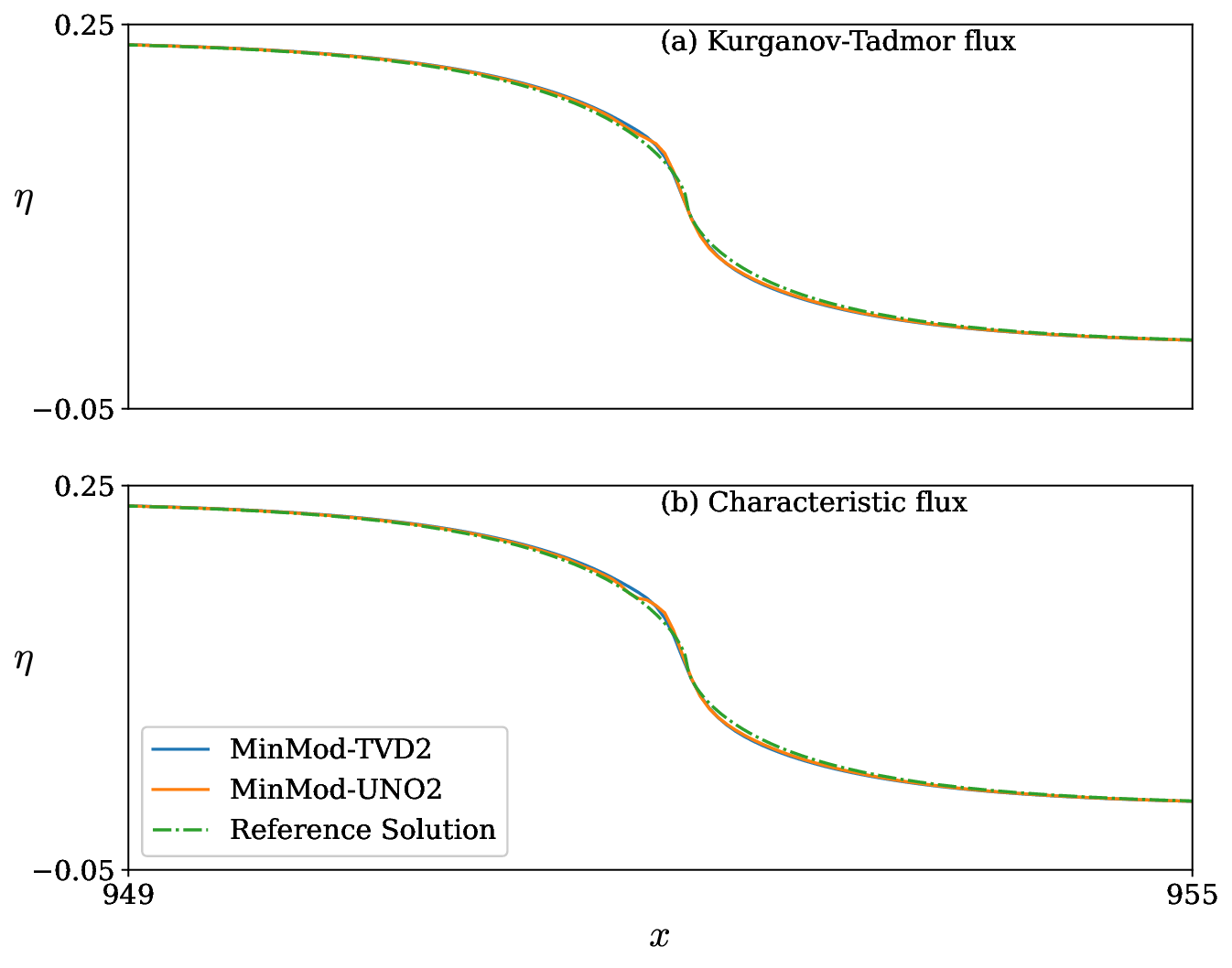}
\caption{Comparison between TVD2 and UNO2 reconstructions with the MinMod limiter and Kurganov-Tadmor flux}
\label{fig:compuno}
\end{figure}

Taking smaller values of $\delta$ we observed that as $\delta\to 0$ the solution converges to the corresponding solution of the shallow water equations. Also all the waves propagate with the same speed. The performance of the finite volume method is satisfactory and no spurious oscillations were observed.

Figure \ref{fig:compfl} presents a comparison of the performance of the Kurganov-Tadmor flux (\ref{KTFlux}) and Characteristic Flux (\ref{CFFlux})  with various limiters with TVD2 reconstruction for the approximation of the weakly singular wave front generated by the initial conditions (\ref{eq:initconddb}) at $t=600$. We also compare the various numerical solutions with a reference (numerical) solution generated by solving Equation (\ref{eq:systemd4}) using the Gauss-Newton approximation of the minimizer of the residual with the Levenberg-Marquardt modification. In the residual equation, the derivatives were evaluated using central finite-difference approximations (with the exception of the boundary cells where one-sided approximations were employed), and no additional handling of either the nonlinear terms or the singularity was applied. The final converged solution exhibited a discrete $L^2$-norm residual on the order of $10^{-5}$.

Both numerical fluxes yield nearly identical results, and no noticeable differences can be detected. Although the MinMod limiter provides the most accurate representation of the weakly singular front, this front is accompanied by low-amplitude oscillations. Note that the oscillations produced are smooth and not spurious but it is rather the approximated profile of the weakly singular shock. This effect is due to their high-resolution of the derivative.

Similarly, the use of UNO2 reconstruction yields comparable results. Figure \ref{fig:compuno} illustrates a comparison of the numerical solutions obtained using TVD2 and UNO2 reconstruction for the MinMod limiter. We note that although the two reconstruction methods produce very similar results, the weakly singular front in the TVD2 reconstruction exhibits smaller oscillations.

Figure \ref{fig:energy1} presents the computed energy $\mathcal{E}$ as a function of time. Note that the integral in the formula of the energy was evaluated using composite midpoint rule, while we only present the case with the TVD2 reconstruction with the MinMod limiter and the Kurganov-Tadmor flux. The other cases yield very similar results. We observe that the energy of the solution for both $\delta=1$ and $\delta=0.1$ approximates the energy of the solutions for $\delta=0$. The energy decreases for both values of $\delta$, but when $\delta=1$, the energy starts decreasing at an earlier time. 
\begin{figure}[t]
\centering
\includegraphics[width=0.8\columnwidth]{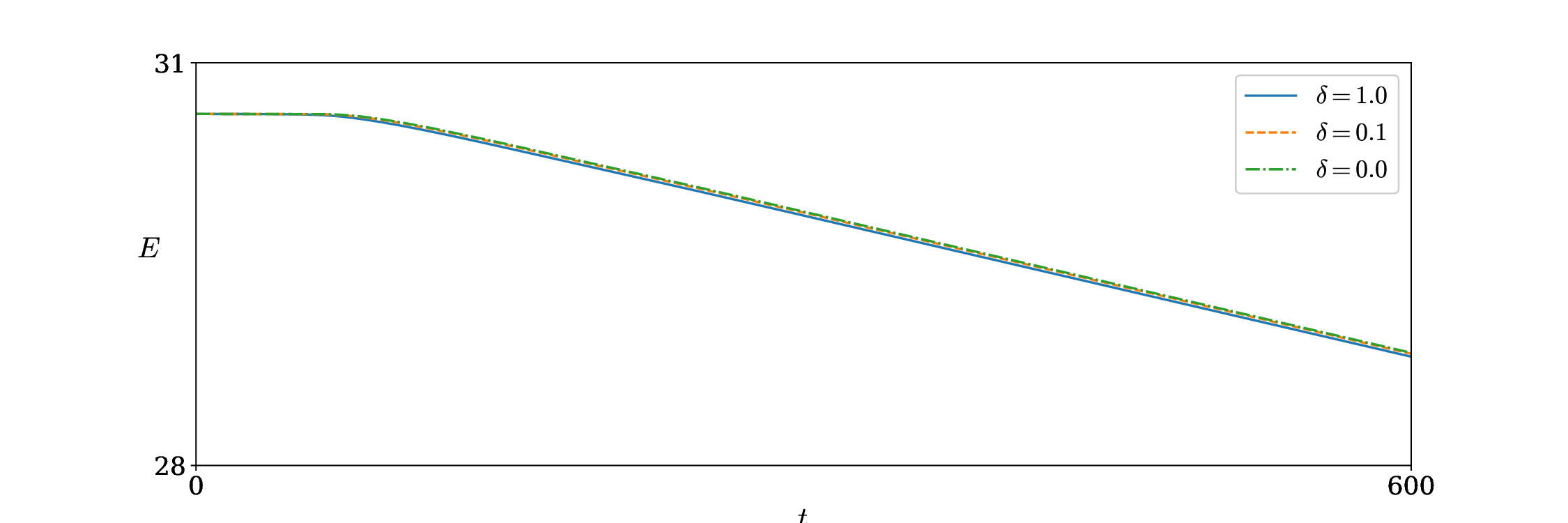}
\caption{Energy $E(t)$ of the solution as a function of time (TVD2 reconstruction with MinMod limiter and Kurganov-Tadmor flux)}
\label{fig:energy1}
\end{figure}

\subsection{Accuracy of the numerical method}\label{sec:accuracy}

The accuracy of this numerical scheme for approximating smooth solutions of similarly structured problems has already been investigated in \cite{DCMM2013,dkm2013,dkm2011}. In contrast, the present work is the first to assess its performance on weakly singular shock waves. Having shown that weakly singular shocks may arise from generic Riemann initial data with sufficiently large mass, we then analyze more closely how well the numerical method captures these solutions. For these tests, we consider as initial condition the numerically generated weakly singular shock wave and we imposed Dirichlet boundary conditions, even though there is currently no supporting theory that ensures well-posedness of the corresponding initial–boundary value problem. Specifically, guided by the structure of these waves, we prescribed $\eta(a,t)=\eta_{-}$, $\eta(b,t)=\eta_{+}$, $u(a,t)=u_{-}$, and $u(b,t)=u_{+}$, where $\eta_{\pm}$ and $u_{\pm}$ are introduced in Sections (\ref{sec:existence}) and (\ref{sec:gwsw}). 

In this study, we also utilize some error indicators. Let $(\tilde{\eta},\tilde{u})$ represent the translated reference cell-averages of the solution, and let $(H,U)$ denote the corresponding numerical approximations obtained through our method. We define the errors as follows:
$$E^1(w;t)=\sum_{i}|\tilde{w}_i-W_i|,\quad E^2(w;t)=\left(\sum_i|\tilde{w}_i-W_i|^2\right)^{1/2},\quad E^\infty(w;t)=\max_i|\tilde{w}_i-W_i|\ ,$$
where $w$ is either $\eta$ or $u$.

We first solve the initial-boundary value problem for $\delta=0.01$ in $[-100,100]$ using $\Delta x=0.1$ and $\Delta t=0.01$ up to $T=100$. As an initial condition, we used the numerically computed weakly singular shock wave translated such as the singularity is located at $x=-50$. This shock wave was obtained by solving Equation (\ref{eq:systemd4}) with $s = 1.17$. The Gauss–Newton method was used and the Levenberg–Marquardt modification was applied to minimize the residual as previously described. This process resulted in a discrete $L^2$-norm residual of order $O(10^{-5})$. The resulting shock wave profile was then mapped onto the finite volume grid using piecewise linear interpolation. This numerical solution served as the reference solution for our comparison. To compare the reference solution with numerical solution generated by the finite volume method at time $t$, we translated the reference solution horizontally by $s\cdot t$ units to the right extending and truncating the constant states appropriately. 

Figure \ref{fig:front} presents the $\eta$ component of the solution ($u$ is very similar and is omitted) for $t=0,25,50,75$ and $100$  along the magnification of shock fronts obtained using the Kurganov-Tadmor flux and TVD2 and UNO2 reconstruction with MinMod limiter. We observe that the front was stabilized very fast to a moving front while the numerical solution exhibits a steeper profile compared to the reference solution, with errors $E^1$, $E^2$, and $E^\infty$ all being of order $O(10^{-3})$. Furthermore, the front of the solution sheds a small amplitude wavelet that propagates to the left, approaching the left boundary at the top of the solution. In this simulation, the TVD2 reconstruction appears to resolve the weakly singular front more accurately than the UNO2 reconstruction. A magnification near the peak of the solution, highlighting the small-amplitude wavelet, is shown in Figure \ref{fig:wavlet}. This indicates that the initial condition does not correspond to an exact discrete traveling wave.

\begin{figure}[t]
\centering
\includegraphics[width=\columnwidth]{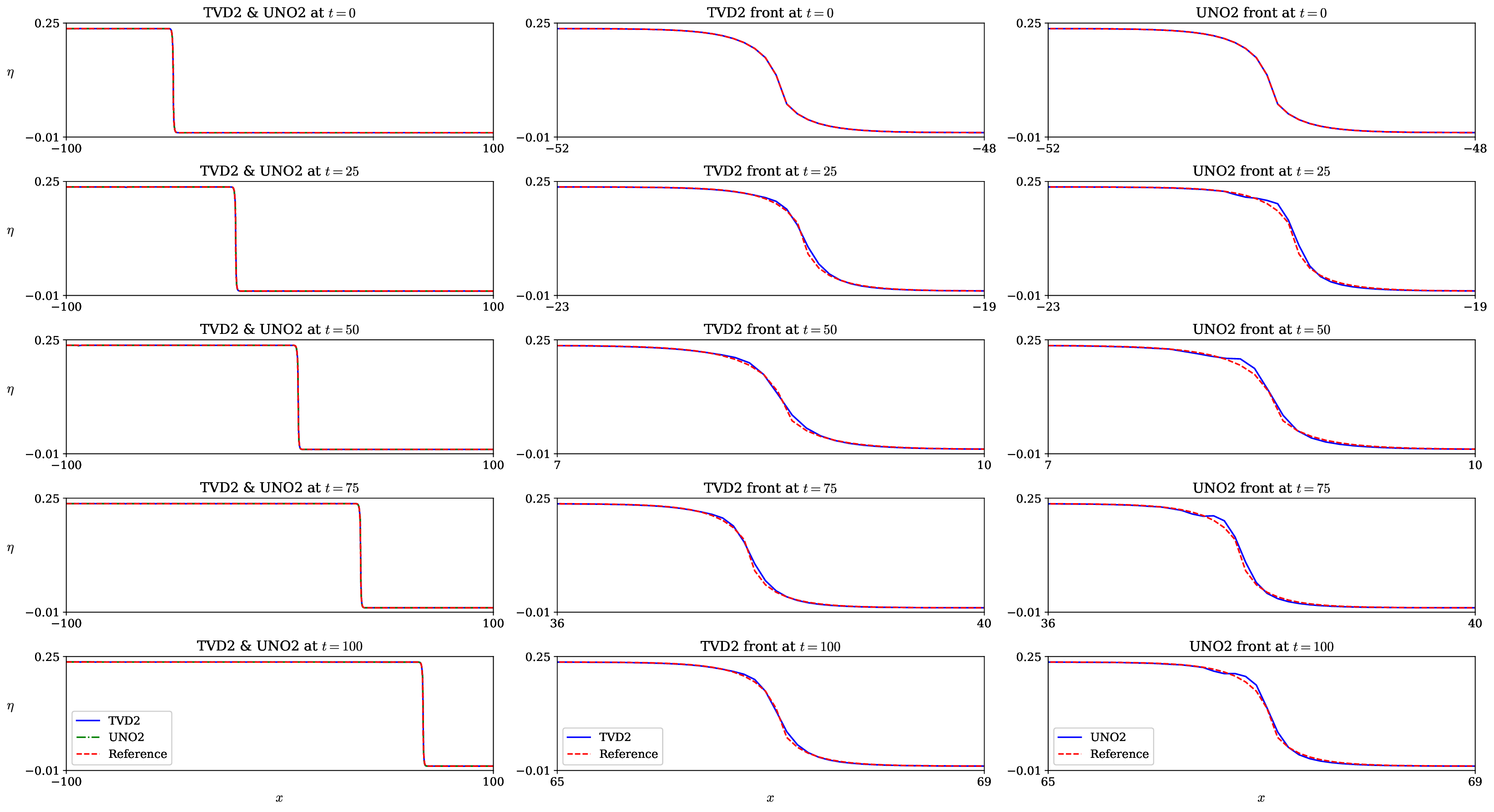}
\caption{Propagation of a weakly singular solitary wave with $s=1.17$, $\delta=0.01$. (TVD2 and UNO2 reconstruction with MinMod limiter and Kurganov-Tadmor flux, $\Delta x=0.1$, $\Delta t=0.01$)}
\label{fig:front}
\end{figure}

\begin{figure}[t]
\centering
\includegraphics[width=0.5\columnwidth]{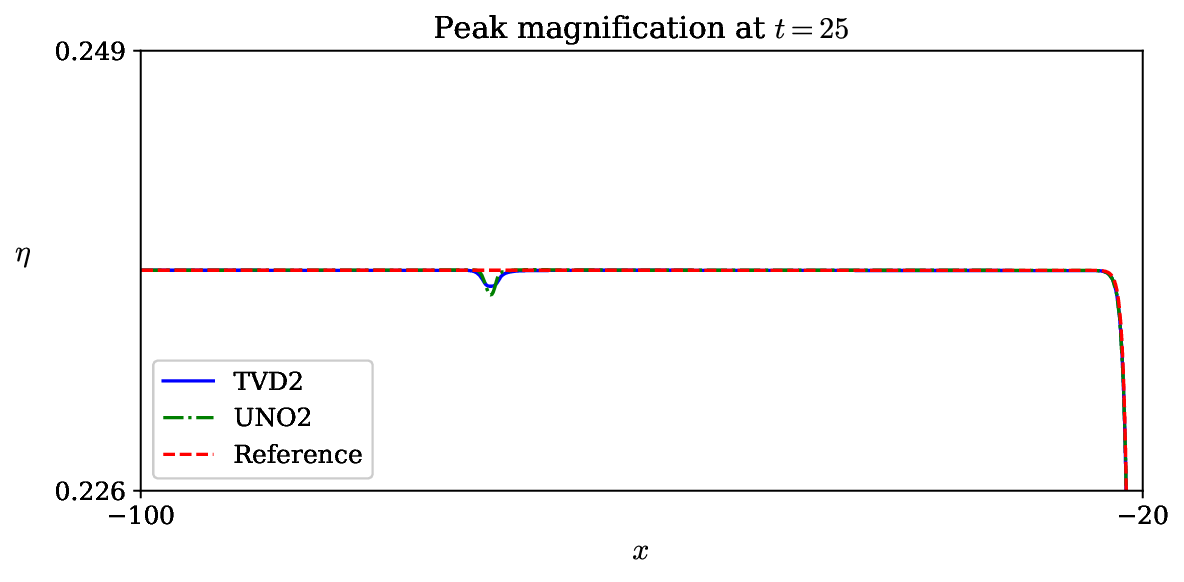}
\caption{Magnification of Figure \ref{fig:front} at $t=25$}
\label{fig:wavlet}
\end{figure}

As $\Delta x$ becomes smaller, the discrepancy between the two solutions and the small sheared wave diminishes. It is worth noting that these profiles also align with the profiles observed in Section \ref{sec:wekshoK}, where the weakly singular shock wave was generated from a generic initial condition. 

Since there is no exact formula for weakly singular shock waves, and the reference solution is also a numerical solution, determining the precise convergence rate is challenging. However, to further assess the accuracy of the method, we evaluated its convergence as $\Delta x$ decreased by simulating the propagation of the same weak-singular shock wave in the interval $[-100, 100]$ up to time $T$ for a decreasing sequence of values of $\Delta x_i$. To ensure that the errors from the time discretization are negligible, we used $\Delta t_i=\Delta x_i/10$. We computed the errors $E^1_i$, $E^2_i$, and $E^{\infty}_i$ for the corresponding mesh length $\Delta x_i$. Subsequently, we determined the experimental order of convergence, which is defined as 
$$
  \text{rate} = \frac{\log(E^j_i / E^j_{i+1})}{\log(\Delta x_i / \Delta x_{i+1})}\ ,
$$
at $T=2$.

\begin{table}[t]
  \centering
  \begin{tabular}{ccccccr}
    \hline
    $\Delta x$  & $E^1(\eta;t=2)$ & $E^1(u;t=2)$ & rate $\eta$ & rate $u$  \\
    \hline
  0.1000 & $7.110840 \times 10^{-3}$ & $6.690155 \times 10^{-3}$ & -- & -- \\
  0.0500 & $5.346618 \times 10^{-3}$ & $4.965604 \times 10^{-3}$ & 0.41 & 0.43 \\
  0.0100 & $3.306735 \times 10^{-3}$ & $2.970371 \times 10^{-3}$ & 0.30 & 0.32 \\
  0.0050 & $2.202325 \times 10^{-3}$ & $1.972385 \times 10^{-3}$ & 0.59 & 0.59 \\
  0.0025 & $1.339119 \times 10^{-3}$ & $1.208850 \times 10^{-3}$ & 0.72 & 0.71 \\
    \hline
  \end{tabular}
  \caption{Errors and experimental order of convergence in discrete $L^1$ norm for a weakly singular solitary wave with $s=1.17$, $\delta=0.01$ (TVD2 reconstruction with MinMod limiter and Kurganov-Tadmor flux)}\label{tab:T1}
  \end{table}
\begin{table}[t]
  \centering
  \begin{tabular}{ccccccr}
    \hline
    $\Delta x$  & $E^2(\eta;t=2)$ & $E^2(u;t=2)$ & rate $\eta$ & rate $u$  \\
    \hline
  0.1000 & $6.709960 \times 10^{-3}$ & $6.354103 \times 10^{-3}$ & -- & -- \\
  0.0500 & $4.214414 \times 10^{-3}$ & $3.903600 \times 10^{-3}$ & 0.67 & 0.70 \\
  0.0100 & $2.519163 \times 10^{-3}$ & $2.336817 \times 10^{-3}$ & 0.32 & 0.32 \\
  0.0050 & $1.645623 \times 10^{-3}$ & $1.535363 \times 10^{-3}$ & 0.61 & 0.61 \\
  0.0025 & $9.699172 \times 10^{-4}$ & $9.095216 \times 10^{-4}$ & 0.76 & 0.76 \\
    \hline
  \end{tabular}
  \caption{Errors and experimental order of convergence in discrete $L^2$ norm for a weakly singular solitary wave with $s=1.17$, $\delta=0.01$ (TVD2 reconstruction with MinMod limiter and Kurganov-Tadmor flux)}\label{tab:T2}
  \end{table}
\begin{table}[t]
  \centering
  \begin{tabular}{ccccccr}
    \hline
    $\Delta x$  & $E^\infty(\eta;t=2)$ & $E^\infty(u;t=2)$ & rate $\eta$ & rate $u$  \\
    \hline
  0.1000 & $1.551467 \times 10^{-2}$ & $1.340331 \times 10^{-2}$ & -- & -- \\
  0.0500 & $7.594579 \times 10^{-3}$ & $9.928812 \times 10^{-3}$ & 1.03 & 0.43 \\
  0.0100 & $5.709704 \times 10^{-3}$ & $4.815639 \times 10^{-3}$ & 0.18 & 0.45 \\
  0.0050 & $4.018024 \times 10^{-3}$ & $3.549725 \times 10^{-3}$ & 0.51 & 0.44 \\
  0.0025 & $2.648714 \times 10^{-3}$ & $2.510502 \times 10^{-3}$ & 0.60 & 0.50 \\
    \hline
  \end{tabular}
  \caption{Errors and experimental order of convergence in discrete $L^\infty$ norm for a weakly singular solitary wave with $s=1.17$, $\delta=0.01$ (TVD2 reconstruction with MinMod limiter and Kurganov-Tadmor flux)}\label{tab:T3}
  \end{table}

\begin{table}[t]
  \centering
  \begin{tabular}{ccccccr}
    \hline
    $\Delta x$  & $E^1(\eta;t=2)$ & $E^1(u;t=2)$ & rate $\eta$ & rate $u$  \\
    \hline
  0.1000 & $8.243964 \times 10^{-3}$ & $8.127002 \times 10^{-3}$ & -- & -- \\
  0.0500 & $5.721676 \times 10^{-3}$ & $5.196879 \times 10^{-3}$ & 0.53 & 0.65 \\
  0.0100 & $2.412570 \times 10^{-3}$ & $2.153890 \times 10^{-3}$ & 0.54 & 0.55 \\
  0.0050 & $1.401646 \times 10^{-3}$ & $1.272022 \times 10^{-3}$ & 0.78 & 0.76 \\
  0.0025 & $7.131612 \times 10^{-4}$ & $6.667732 \times 10^{-4}$ & 0.97 & 0.93 \\
    \hline
  \end{tabular}
  \caption{Errors and experimental order of convergence in discrete $L^1$ norm for a weakly singular solitary wave with $s=1.17$, $\delta=0.01$ (UNO2 reconstruction with MinMod limiter and Kurganov-Tadmor flux)}\label{tab:T4}
  \end{table}
\begin{table}[t]
  \centering
  \begin{tabular}{ccccccr}
    \hline
    $\Delta x$  & $E^2(\eta;t=2)$ & $E^2(u;t=2)$ & rate $\eta$ & rate $u$  \\
    \hline
  0.1000 & $7.765414 \times 10^{-3}$ & $6.091341 \times 10^{-3}$ & -- & -- \\
  0.0500 & $5.258403 \times 10^{-3}$ & $4.345916 \times 10^{-3}$ & 0.56 & 0.49 \\
  0.0100 & $1.890633 \times 10^{-3}$ & $1.704965 \times 10^{-3}$ & 0.64 & 0.58 \\
  0.0050 & $1.043506 \times 10^{-3}$ & $9.583431 \times 10^{-4}$ & 0.86 & 0.83 \\
  0.0025 & $5.271910 \times 10^{-4}$ & $4.880188 \times 10^{-4}$ & 0.99 & 0.97 \\
    \hline
  \end{tabular}
  \caption{Errors and experimental order of convergence in discrete $L^2$ norm for a weakly singular solitary wave with $s=1.17$, $\delta=0.01$ (UNO2 reconstruction with MinMod limiter and Kurganov-Tadmor flux)}\label{tab:T5}
  \end{table}
\begin{table}[t]
  \centering
  \begin{tabular}{ccccccr}
    \hline
    $\Delta x$  & $E^\infty(\eta;t=2)$ & $E^\infty(u;t=2)$ & rate $\eta$ & rate $u$  \\
    \hline
  0.1000 & $1.462075 \times 10^{-2}$ & $1.125463 \times 10^{-2}$ & -- & -- \\
  0.0500 & $1.398974 \times 10^{-2}$ & $9.306038 \times 10^{-3}$ & 0.06 & 0.27 \\
  0.0100 & $7.225315 \times 10^{-3}$ & $5.558449 \times 10^{-3}$ & 0.41 & 0.32 \\
  0.0050 & $4.556731 \times 10^{-3}$ & $4.221170 \times 10^{-3}$ & 0.67 & 0.40 \\
  0.0025 & $3.299727 \times 10^{-3}$ & $2.953141 \times 10^{-3}$ & 0.47 & 0.52 \\
    \hline
  \end{tabular}
  \caption{Errors and experimental order of convergence in discrete $L^\infty$ norm for a weakly singular solitary wave with $s=1.17$, $\delta=0.01$ (UNO2 reconstruction with MinMod limiter and Kurganov-Tadmor flux)}\label{tab:T6}
  \end{table}

Tables~\ref{tab:T1}--\ref{tab:T3} and \ref{tab:T4}--\ref{tab:T6} report the errors and experimental orders of convergence at $T=2$ obtained with the Kurganov--Tadmor flux and the MinMod limiter, for the TVD2 and UNO2
reconstructions respectively. The remaining combinations of fluxes and limiters produced quantitatively similar results and are not reproduced here.

A useful benchmark for context is the classical convergence theory for finite-volume schemes applied to scalar conservation laws with discontinuous (bounded variation) solutions. The general $L^1$ bound is $O(\sqrt{\Delta x})$, due to Kuznetsov \cite{Kuz1976} and known to be sharp \cite{TT1995,Sab1997} for monotone schemes against generic bounded variation data. Sharper $O(\Delta x)$ rates are recovered for piecewise smooth solutions containing finitely many shocks \cite{TZ1997}, while the pointwise error remains $O(1)$ across each discontinuity \cite{NTT1994}. The values $1/2$ and $1$ thus bracket the standard reference range for finite-volume approximations of genuine shock
discontinuities.

The convergence behavior observed in Tables~\ref{tab:T1}--\ref{tab:T6} departs from this benchmark in a structured way. In the discrete $L^1$ and $L^2$ norms, the experimental order is somewhat erratic on coarse grids but
trends upward under refinement, eventually exceeding $1/2$ in all cases. On the finest mesh, the UNO2 reconstruction reaches approximately $1$, while TVD2 reaches approximately $0.7$--$0.8$ over the same range. The difference
between the two reconstructions is consistent with their design: UNO2 retains second-order accuracy at local extrema \cite{HaOs}, whereas TVD limiters reduce to first order there \cite{Sweby1984}, and the singular
point of a weakly singular front is precisely such an extremum of the derivative. In both cases the asymptotic rate stays above the $1/2$ benchmark of the discontinuous-shock theory, which is consistent with the
front being continuous.

In the discrete $L^\infty$ norm, by contrast, the rate fluctuates around $1/2$ for both reconstructions, in line with the H\"older $C^{0,2/3}$ regularity at the cusp. Nevertheless, $E^\infty$ does decrease with
$\Delta x$ in every case, which is itself a numerical signature of the continuity of the front. For a genuine discontinuity the $L^\infty$ error would remain $O(1)$ and would not converge at all. Qualitatively identical
patterns were observed for other parameter values, for instance $\delta=1$. We emphasize that the reference solution is itself numerical and not exact, so the rates reported here should be regarded as indicative and warrant
further investigation.

We close with a remark on the boundary conditions. The Dirichlet boundary conditions used in this convergence study performed without issue. Nevertheless, because the experiments that follow involve reflections and
wave interactions, periodic boundary conditions are more convenient, and we adopt them throughout the remainder of this section.

\subsection{Interaction of weakly singular shock waves}

Thus far, we have considered the propagation of isolated  weakly singular shock waves. In many physical scenarios, however, sometimes shocks do not evolve in isolation but rather interact with one another or with other wave families. We now turn to a closer examination of these interactions and in particular the head-on collision of such singular fronts. 

\begin{figure}[t]
\centering
\includegraphics[width=\columnwidth]{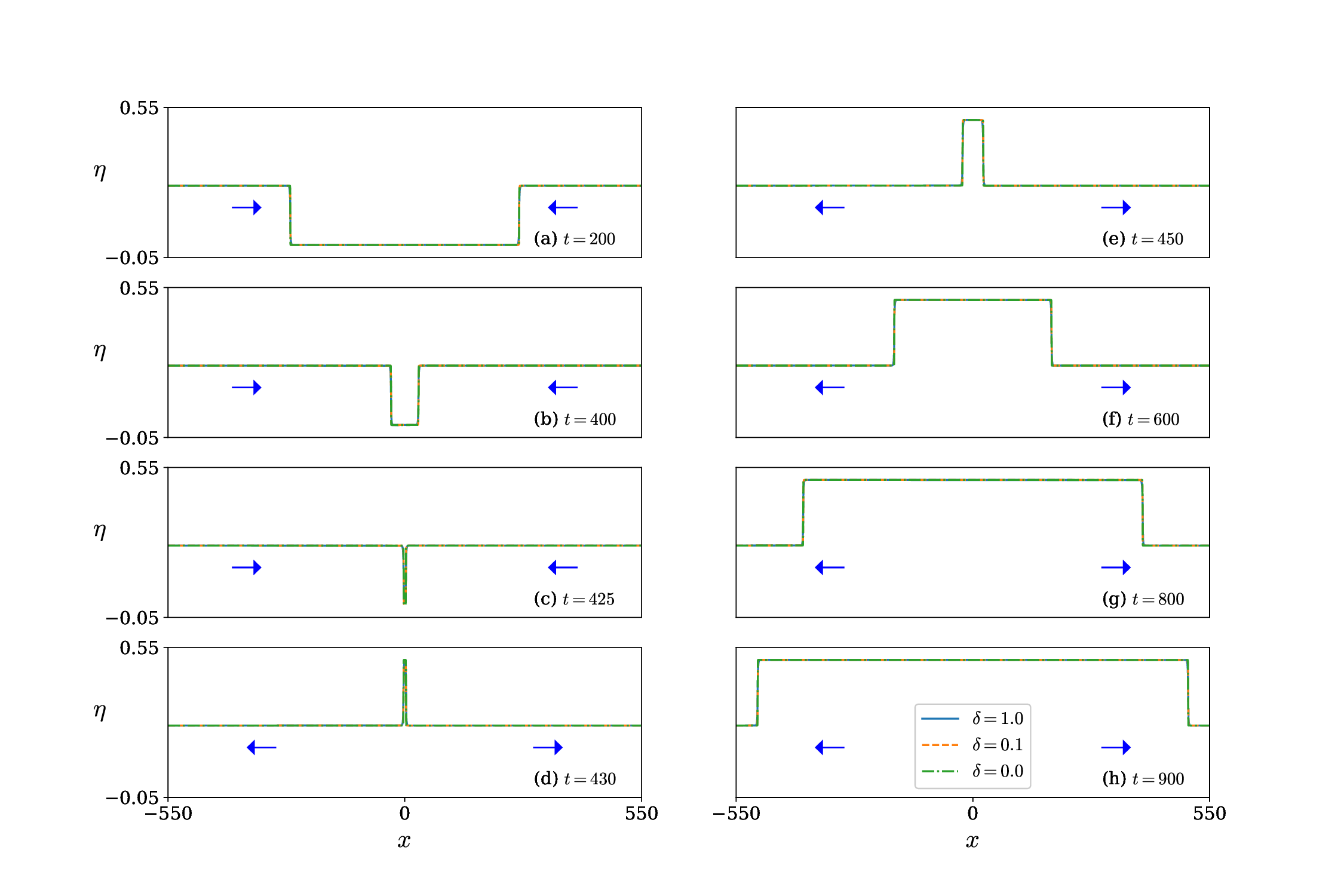}
\caption{Head-on collision of weakly singular shock waves $\delta=0, 0.1$ and $1$ (TVD2 reconstruction with MinMod limiter and Kurganov-Tadmor flux)}
\label{fig:headon0}
\end{figure}

\begin{figure}[t]
\centering
\includegraphics[width=0.7\columnwidth]{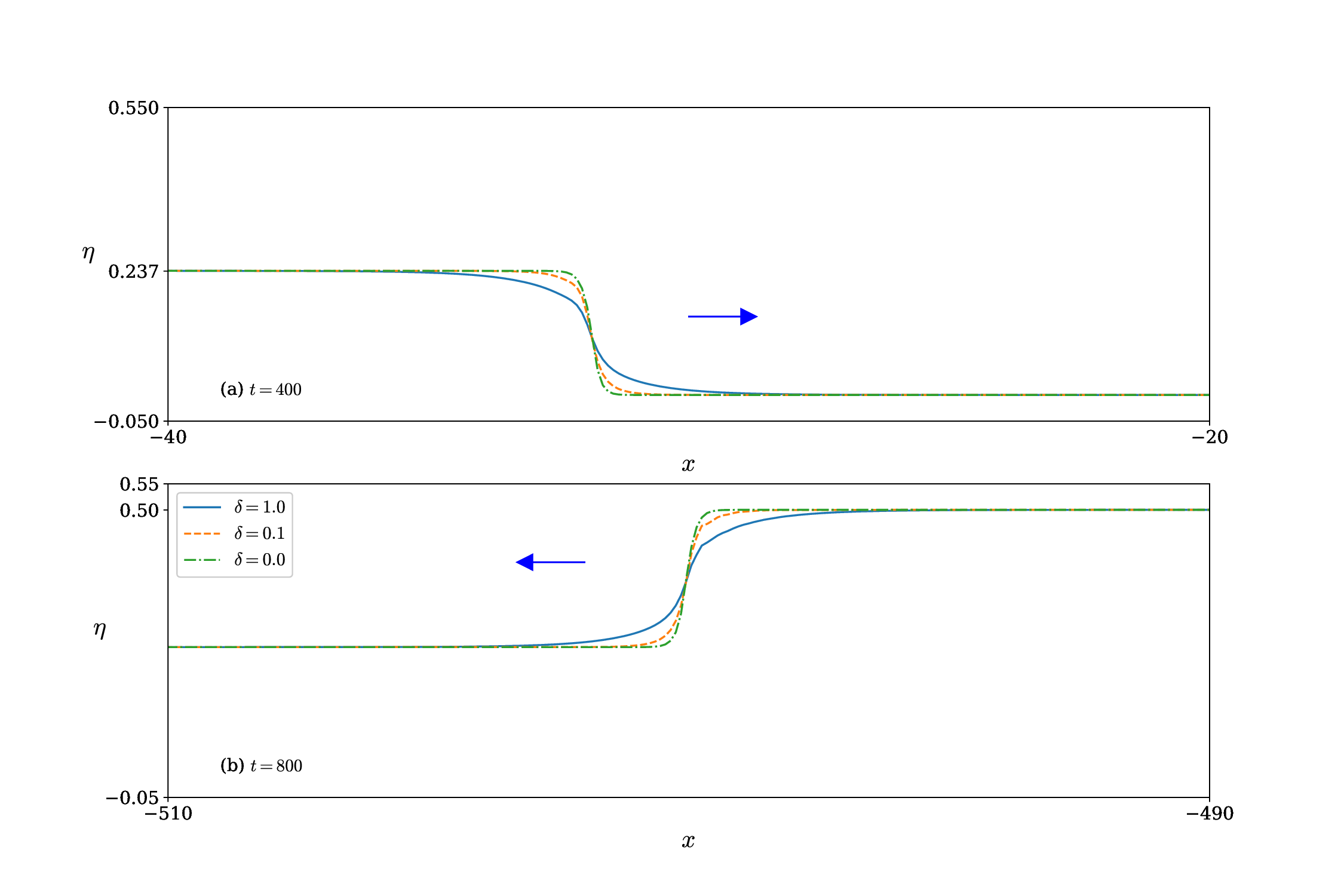}
\caption{Magnification of Figure \ref{fig:headon0} around the wave front}
\label{fig:headon0zoom}
\end{figure}

In this experiment, we consider two $tanh$-profiles, similar to (\ref{eq:initconddb}), oriented in opposite directions as shown in Figure \ref{fig:headon0}(a). This initial condition generates two weakly singular shock waves of the same amplitude as before ($\eta=0.237549$), which propagate in opposite directions and eventually collide. Figure \ref{fig:headon0} illustrates the collision of these waves for $\delta = 0$, $0.1$, and $1$. During the interaction, which takes place in the time interval $(400,450)$, the two weakly singular fronts merge and give rise to two new weakly singular fronts. These propagate on top of the original waves but in opposite directions. The interaction is nonlinear, as the amplitude of the resulting waves increases to $\eta_{-} = 0.500318$ while $\eta_{+}=0.237549$. Moreover, we have $u_{-}=0$ and $u_{+}=-0.224586$. Therefore, the speed of the right-traveling shock wave from (\ref{eq:equilib}) is approximately $1.058$. As expected, the interaction is non-dispersive, and no dispersive oscillations are observed. The arrows in Figure \ref{fig:headon0} indicate the propagation direction of the corresponding wave fronts.

Although no visible differences can be observed among the various solutions in Figure \ref{fig:headon0}, a closer inspection in Figure \ref{fig:headon0zoom}, which shows a magnification of the wave front for $x<0$, reveals that the traveling fronts become smoother as $\delta$ increases. Furthermore, the singular point is preserved for all values of $\delta$, within numerical accuracy, in agreement with the theoretical discussion presented earlier. Moreover, similar very small oscillations are generated on the smooth front due to the singularity in the derivative. These are visible mainly for $\delta=1$. This experiment, together with the previous ones, also demonstrates the stability of weakly singular shock waves.

\subsection{Generation of localized solutions}

In this section we integrate numerically system (\ref{eq:bous}) with initial condition 
\begin{equation}\label{eq:initcond2}
    \eta(x,0)=e^{-x^2/10},\quad u(x,0)=0\ ,
\end{equation}
for $x \in [-1000,1000]$ under periodic boundary conditions. Simulations were carried out for $\delta = 0, 0.1,$ and $1$. This initial condition generates two waves propagating in opposite directions, which eventually break and form either a weakly singular front (when $\delta \neq 0$) or a classical shock wave (when $\delta = 0$).

\begin{figure}[t]
\centering
\includegraphics[width=0.7\columnwidth]{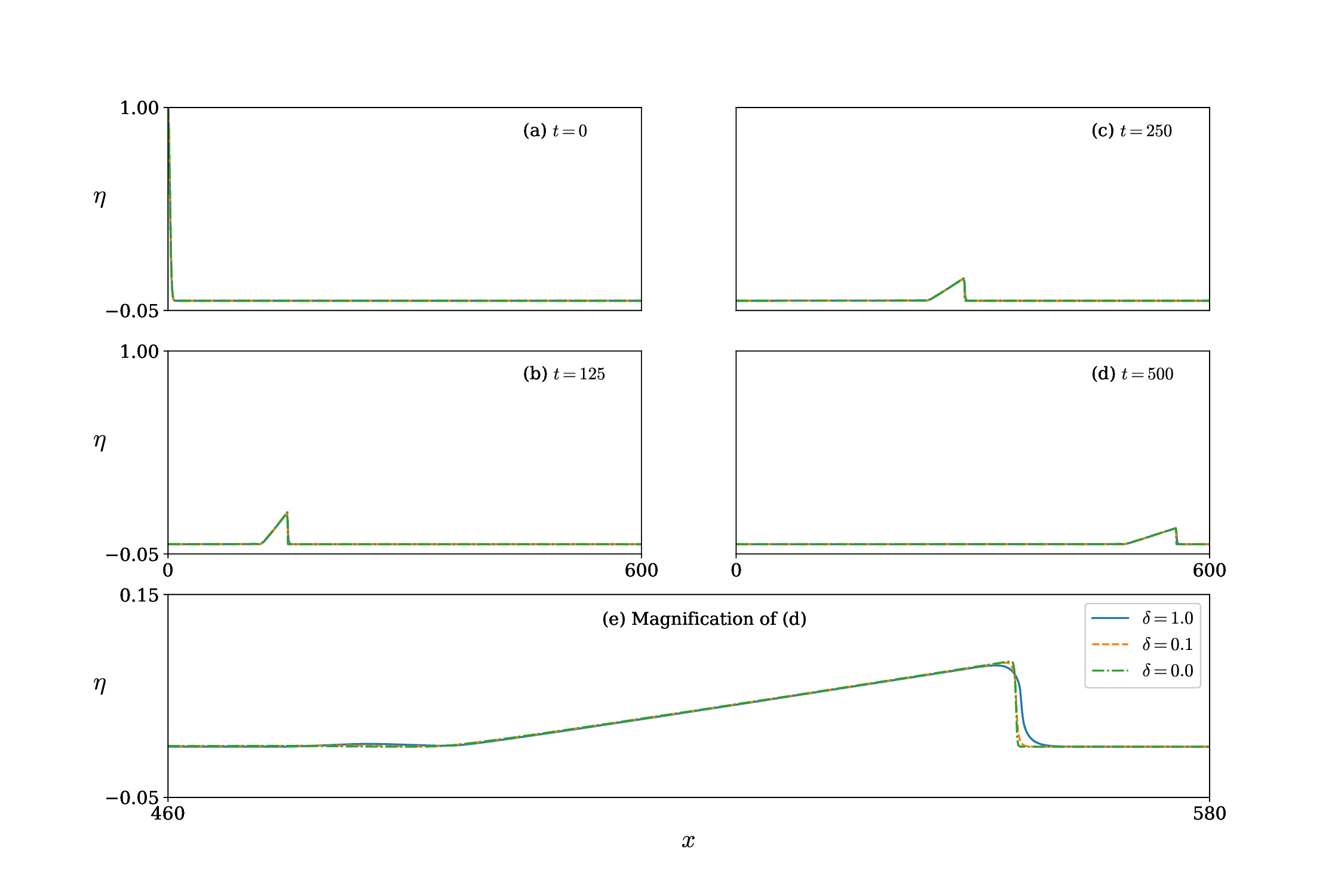}
\caption{Generation of weakly singular localized wave for $\delta=1$, $0.1$ and comparison with classical shock wave of $\delta=0$ (TVD2 reconstruction with MinMod limiter and Kurganov-Tadmor flux)}
\label{fig:localize}
\end{figure}

\begin{figure}[t]
\centering
\includegraphics[width=0.7\columnwidth]{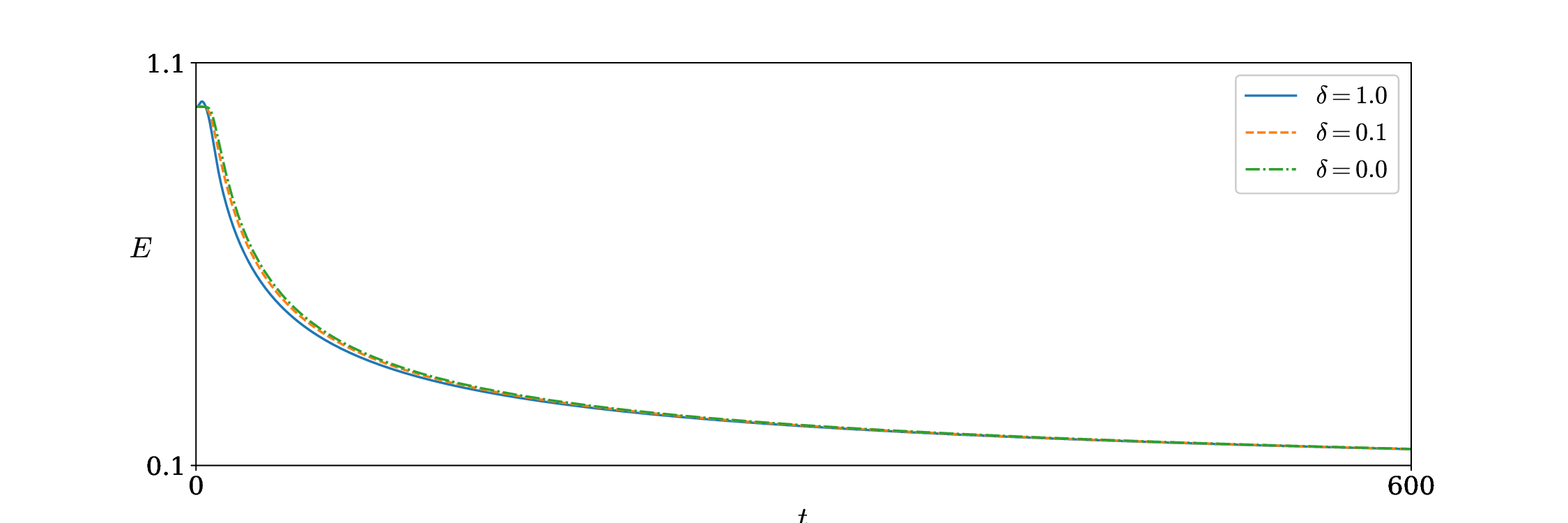}
\caption{Energy $E(t)$ of the solution as a function of time (TVD2 reconstruction with MinMod limiter and Kurganov-Tadmor flux)}
\label{fig:energy2}
\end{figure}

Figure \ref{fig:localize} illustrates the formation of these weakly singular wave fronts. As before, we present only half of the computational domain, since the other half is symmetric. A magnification of the wave front is shown in Figure \ref{fig:localize}(e), highlighting the differences among the various values of $\delta$. Once again, the front becomes progressively less steep as $\delta$ increases. It is also worth noting that, during the formation of the weakly singular wave, a small-amplitude trailing tail emerges. The size of this tail grows with increasing $\delta$.
Furthermore, since these waves are not traveling, differences in the phase of the singular fronts can be observed: the phase speed of the wave front increases with larger values of $\delta$. In contrast, as $\delta \to 0$, the weakly singular solution appears to converge toward the discontinuous solution of the classical shallow water equations.

This experiment demonstrates that cusped solitons, which are possible solutions of the inviscid regularization of the Boussinesq system (\ref{eq:bous}), may not be stable. In general, they cannot be generated from arbitrary initial conditions and do not propagate in a stable manner when used as initial conditions themselves. Moreover, the solutions of (\ref{eq:bous}) appear to follow the same qualitative patterns as those of the shallow water equations (\ref{eq:sw1}).

Finally, Figure \ref{fig:energy2} presents the computed energy $E(t)$ for the three different cases of $\delta$. Again, the energy for $\delta=1$ and $\delta=0.1$ approximate the energy of the solution for $\delta=0$. However, in this case the energy for $\delta>0$ is increasing initially, showing that the system is not in general dissipative. Also, we observe that the dissipation rate is proportional to the value of $\delta$. Thus larger values of $\delta$ result in smaller energy for large times. 

\subsection{Interaction of weakly singular shock wave with a sudden depth transition}\label{sec:bottom}

We close this work with a study of the system in the presence of a variable bottom topography. In particular, we perform a comparison between the interaction of a weakly singular shock wave of the regularized system and the shallow water equations with steep bottom variations. In this experiment we use the same initial condition as in Section \ref{sec:wekshoK}, while the depth function was chosen such as
$$D(x)= 2-\frac{1}{2}\left[2-\tanh\left(\tfrac{1}{2}(x-900)\right)+\tanh\left(\tfrac{1}{2}(x-500)\right)\right]\ .$$ 
The initial profile and the bottom topography are presented in Figure \ref{fig:bottom}. For the numerical solution we used both Kurganov-Tadmor and Characteristic flux with $\Delta x=0.05$ and $\Delta t=0.025$, however because the results are almost identical, we only present the results obtained with the Kurganov-Tadmor flux with the TVD2-MinMod limiter. In these experiments we took once again $g=1$.

\begin{figure}[t]
\centering
\includegraphics[width=0.6\columnwidth]{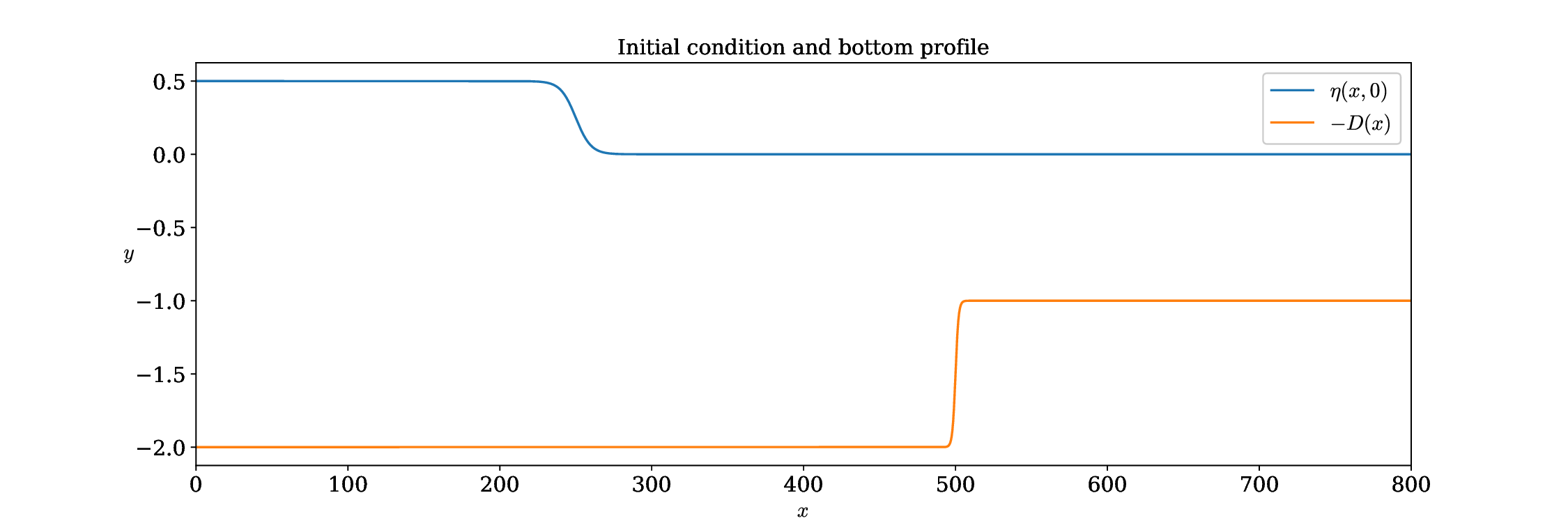}
\caption{Initial condition and bottom profile}
\label{fig:bottom}
\end{figure}
\begin{figure}[ht!]
\centering
\includegraphics[width=1\columnwidth]{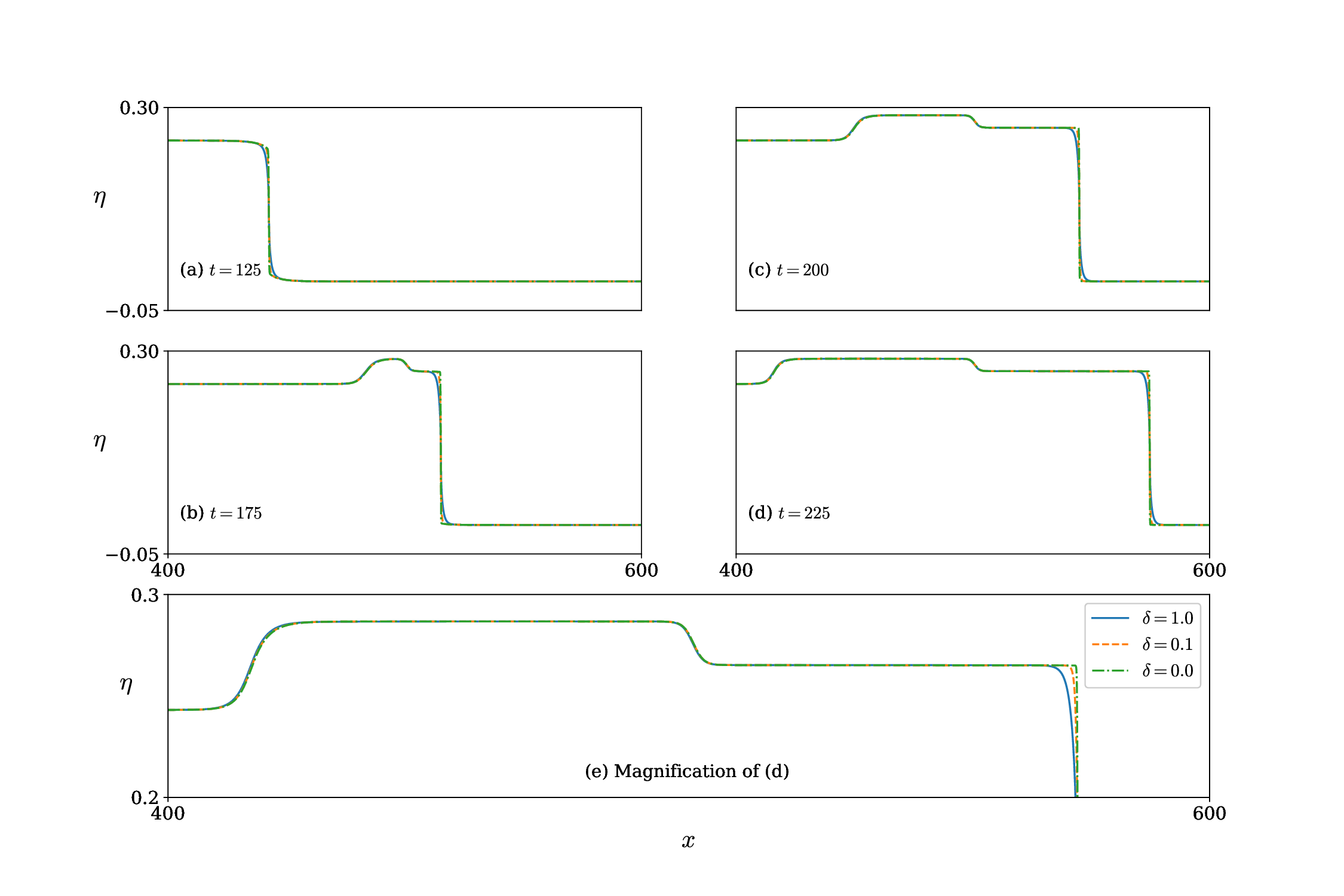}
\caption{Interaction of weakly singular shock wave with steep gradient bottom topography (TVD2 reconstruction with MinMod limiter and Kurganov-Tadmor flux)}
\label{fig:botinterc}
\end{figure}

After the shock wave is formed, it interacts with the steep bottom topography and part of it is reflected back while a new shock wave emerges with higher amplitude. This is depicted in Figure \ref{fig:botinterc}, where we observe that the solutions of the two models agree at all points, including in the speeds of the waves and the connecting states. However, as expected the solutions with a higher value for the regularization parameter $\delta$ have smoother shock front compared to the others with smaller value $\delta$. Impressively, the reflecting wave appears to be almost identical, indicating that this approximation is more accurate compared to the Hamiltonian regularization of \cite{CDM2019}. 

\section{Conclusions}\label{sec:conclusions}

In this study we investigated a non-dispersive, non-dissipative regularization of the shallow water equations that admits weakly singular traveling-wave solutions, which are continuous profiles possessing an unbounded derivative at an isolated point. Using a dynamical-systems analysis of the profile equations, we constructed such solutions, both weakly singular shock waves connecting two distinct states and cusped solitons, by concatenating phase-plane orbits across the interior singularity of the flow. We then placed this construction on a rigorous footing, proving that the resulting profiles are weak solutions of the traveling-wave equations in the sense of distributions, the apparent obstruction posed by the singular derivative being removed by the continuity of the momentum flux. On the level of the full system, we further showed that, for small $\delta$ and as long as the solution remains smooth, the energy of the regularized model stays within $O(\sqrt{\delta}\,t)$ of that of the hyperbolic shallow water equations.

To study these solutions dynamically, we developed and compared several finite volume methods that combine the Kurganov--Tadmor and characteristic fluxes with various limiters and TVD2/UNO2 reconstructions. Applying these methods to numerically approximate weakly singular traveling fronts of the regularized system, we found that the TVD2 reconstruction achieves slightly better resolution compared to UNO2 for these particular weakly singular fronts. The simulations showed that this regularization reproduces the dynamics of the nonlinear shallow water model accurately, including the formation of classical shock waves in the limit $\delta\to0$. The weakly singular shock waves proved numerically to be dynamically stable as they emerged from generic smooth initial data connecting different states, survived head-on interactions, and interacted with steep bottom topography much as their hyperbolic counterparts do. By contrast, we were unable to generate cusped solitons from generic initial conditions, which suggests that these structures are unstable. Finally, the experiments confirmed that solutions of the regularized system neither conserve nor monotonically dissipate the shallow-water energy in general, energy dissipation sets in precisely when a weakly singular front develops, consistently with the behavior of the hyperbolic shallow water equations.

These results indicate that weakly singular shock waves arise not from a Hamiltonian structure but from the interplay of high-order terms with the non-dispersive character of the equations, and that finite volume methods are well suited to their approximation.

\bibliographystyle{plain} 
\bibliography{biblio}

\end{document}